\documentclass[a4paper,leqno]{article}

\usepackage{amsmath}
\usepackage{amsthm}
\usepackage{amssymb}
\usepackage{hyperref}
\usepackage{caption}
\usepackage{subcaption}
\usepackage{graphicx}
\usepackage{lineno}

\newtheorem{theorem}{Theorem}

\newtheorem{proposition}[theorem]{Proposition}

\newtheorem{conjecture}[theorem]{Conjecture}

\newtheorem{definition}[theorem]{Definition}
\newtheorem{problem}[theorem]{Problem}
\newtheorem{algorithm}[theorem]{Algorithm}
\newtheorem{procedure}[theorem]{Procedure}

\newtheorem{remark}[theorem]{Remark}
\numberwithin{equation}{section}
\numberwithin{theorem}{section}

\newcommand{\abs}[1]{|#1|}
\newcommand{\set}[1]{\left\{#1\right\}}

\newcommand{\arrowsv}[0]{\overset{v}{\rightarrow}}

\newcommand{\mH}[0]{\mathcal{H}}

\newcommand{\uni}[2]{{#1}\big\vert_{#2}}

\newcommand{\wH}[3]{\widetilde{\mH}(\uni{#1}{#2}; {#3})}
\newcommand{\wFv}[3]{\widetilde{F}_v(\uni{#1}{#2}; {#3})}

\newcommand{\rp}[0]{{r'_0}}
\newcommand{\rpp}[0]{{r''_0}}

\DeclareMathOperator{\G}{G}
\DeclareMathOperator{\V}{V}
\DeclareMathOperator{\E}{E}

\modulolinenumbers[1]

\begin{document}


\title{On the vertex Folkman numbers\\ $F_v(a_1, ..., a_s; m - 1)$\\ when $\max\{a_1, ..., a_s\} = 6$ or $7$}

\author{
	Aleksandar Bikov\thanks{Corresponding author} \hfill Nedyalko Nenov\\
	\let\thefootnote\relax\footnote{Email addresses: \texttt{asbikov@fmi.uni-sofia.bg}, \texttt{nenov@fmi.uni-sofia.bg}}\\
	\smallskip\\
	Faculty of Mathematics and Informatics\\
	Sofia University "St. Kliment Ohridski"\\
	5, James Bourchier Blvd.\\
	1164 Sofia, Bulgaria
}

\maketitle

\begin{abstract}
Let $G$ be a graph and $a_1, ..., a_s$ be positive integers. The expression $G \overset{v}{\rightarrow} (a_1, ..., a_s)$ means that for every coloring of the vertices of $G$ in $s$ colors there exists $i \in \{1, ..., s\}$ such that there is a monochromatic $a_i$-clique of color $i$. The vertex Folkman numbers $F_v(a_1, ..., a_s; q)$ are defined by the equality:
$$
F_v(a_1, ..., a_s; q) = \min\{|V(G)| : G \overset{v}{\rightarrow} (a_1, ..., a_s) \mbox{ and } K_q \not\subseteq G\}.
$$
Let $m = \sum\limits_{i=1}^s (a_i - 1) + 1$. It is easy to see that $F_v(a_1, ..., a_s; q) = m$ if $q \geq m + 1$. In [11] it is proved that $F_v(a_1, ..., a_s; m) = m + \max\{a_1, ..., a_s\}$. We know all the numbers $F_v(a_1, ..., a_s; m - 1)$ when $\max\{a_1, ..., a_s\} \leq 5$ and none of these numbers is known if $\max\{a_1, ..., a_s\} \geq 6$. In this paper we compute all the numbers $F_v(a_1, ..., a_s; m - 1)$ when $\max\{a_1, ..., a_s\} = 6$. We also prove that $F_v(2, 2, 7; 8) = 20$ and obtain new bounds on the numbers $F_v(a_1, ..., a_s; m - 1)$ when $\max\{a_1, ..., a_s\} = 7$.

\bigskip\emph{Keywords: } Folkman number, clique number, independence number, chromatic number
\end{abstract}



\section{Introduction}

Only finite, non-oriented graphs without loops and multiple edges are considered in this paper. $G_1 + G_2$ denotes the graph $G$ for which $\V(G) = \V(G_1) \cup \V(G_2)$ and $\E(G) = \E(G_1) \cup \E(G_2) \cup E'$, where $E' = \set{[x, y] : x \in \V(G_1), y \in \V(G_2)}$, i.e. $G$ is obtained by connecting every vertex of $G_1$ to every vertex of $G_2$. All undefined terms can be found in \cite{W01}. 

Let $a_1, ..., a_s$ be positive integers. The expression $G \arrowsv (a_1, ..., a_s)$ means that for any coloring of $\V(G)$ in $s$ colors ($s$-coloring) there exists $i \in \set{1, ..., s}$ such that there is a monochromatic $a_i$-clique of color $i$. In particular, $G \arrowsv (a_1)$ means that $\omega(G) \geq a_1$. Further, for convenience, instead of $G \arrowsv (\underbrace{2, ..., 2}_r)$ we write $G \arrowsv (2_r)$ and instead of $G \arrowsv (\underbrace{2, ..., 2}_r, a_1, ..., a_s)$ we write $G \arrowsv (2_r, a_1, ..., a_s)$.

Define:

$\mH(a_1, ..., a_s; q) = \set{ G : G \arrowsv (a_1, ..., a_s) \mbox{ and } \omega(G) < q }.$

$\mH(a_1, ..., a_s; q; n) = \set{ G : G \in \mH(a_1, ..., a_s; q) \mbox{ and } \abs{\V(G)} = n }.$

The vertex Folkman number $F_v(a_1, ..., a_s; q)$ is defined by the equality:
\begin{equation*}
F_v(a_1, ..., a_s; q) = \min\set{\abs{\V(G)} : G \in \mH(a_1, ..., a_s; q)}.
\end{equation*}

The graph $G$ is called an extremal graph in $\mH(a_1, ..., a_s; q)$ if $G \in \mH(a_1, ..., a_s;q )$ and $\abs{\V(G)} = F_v(a_1, ..., a_s; q)$. We denote by $\mH_{extr}(a_1, ..., a_s; q)$ the set of all extremal graphs in $\mH(a_1, ..., a_s; q)$.

We say that $G$ is a maximal graph in $\mH(a_1, ..., a_s; q)$ if $G \in \mH(a_1, ..., a_s; q)$ but $G + e \not\in \mH(a_1, ..., a_s; q), \forall e \in \E(\overline{G})$, i.e. $\omega(G + e) = q, \forall e \in \E(\overline{G})$. $G$ is a minimal graph in $\mH(a_1, ..., a_s; q)$ if $G \in \mH(a_1, ..., a_s; q)$ but $G - e \not\in \mH(a_1, ..., a_s; q), \forall e \in \E(G)$, i.e. $G - e \not\arrowsv (a_1, ..., a_s), \forall e \in \E(G)$.

For convenience, we also define the following term:
\begin{definition}
\label{definition: (+K_t)}
The graph $G$ is called a $(+K_t)$-graph if $G + e$ contains a new $t$-clique for all $e \in \E(\overline{G})$.
\end{definition}
Obviously, $G \in \mH(a_1, ..., a_s; q)$ is a maximal graph in $\mH(a_1, ..., a_s; q)$ if and only if $G$ is a $(+K_q)$-graph. We shall denote by $\mH_{+K_t}(a_1, ..., a_s; q)$ the set of all $(+K_t)$-graphs in $\mH(a_1, ..., a_s; q)$, and by $\mH_{max}(a_1, ..., a_s; q)$ all maximal $K_q$-free graphs in this set.

\begin{remark}
\label{remark: mathcal(H)(a_1; q; n) = ...}
In the special case $s = 1$ we have

$\mH(a_1; q; n) = \set{ G : a_1 \leq \omega(G) < q \mbox{ and } \abs{\V(G)} = n }$.

If $a_1 \leq n \leq q - 1$ then $K_n \in \mH_{max}(a_1; q; n)$, and if $n \geq q - 1 \geq a_1$, then $\mH_{max}(a_1; q; n) = \mH_{max}(q - 1; q; n).$
\end{remark}

Folkman proves in \cite{Fol70} that:
\begin{equation}
\label{equation: F_v(a_1, ..., a_s; q) exists}
F_v(a_1, ..., a_s; q) \mbox{ exists } \Leftrightarrow q > \max\set{a_1, ..., a_s}.
\end{equation}
Other proofs of (\ref{equation: F_v(a_1, ..., a_s; q) exists}) are given in \cite{DR08} and \cite{LRU01}. In the special case $s = 2$, a very simple proof of this result is given in \cite{Nen85} with the help of corona product of graphs.\\
Obviously $F_v(a_1, ..., a_s; q)$ is a symmetric function of $a_1, ..., a_s$, and if $a_i = 1$, then
\begin{equation*}
F_v(a_1, ..., a_s; q) = F_v(a_1, ..., a_{i-1}, a_{i+1}, ..., a_s; q).
\end{equation*}
Therefore, it is enough to consider only such Folkman numbers $F_v(a_1, ..., a_s; q)$ for which
\begin{equation}
\label{equation: 2 leq a_1 leq ... leq a_s}
2 \leq a_1 \leq ... \leq a_s.
\end{equation}
We call the numbers $F_v(a_1, ..., a_s; q)$ for which the inequalities (\ref{equation: 2 leq a_1 leq ... leq a_s}) hold canonical vertex Folkman numbers.\\
In \cite{LU96} for arbitrary positive integers $a_1, ..., a_s$ the following terms are defined
\begin{equation}
\label{equation: m and p}
m(a_1, ..., a_s) = m = \sum\limits_{i=1}^s (a_i - 1) + 1 \quad \mbox{ and } \quad p = \max\set{a_1, ..., a_s}.
\end{equation}
It is easy to see that $K_m \arrowsv (a_1, ..., a_s)$ and $K_{m - 1} \not\arrowsv (a_1, ..., a_s)$. Therefore
\begin{equation*}
F_v(a_1, ..., a_s; q) = m, \quad q \geq m + 1.
\end{equation*}
The following theorem for the numbers $F_v(a_1, ..., a_s; m)$ is true:
\begin{theorem}
\label{theorem: F_v(a_1, ..., a_s; m) = m + p}
Let $a_1, ..., a_s$ be positive integers and let $m$ and $p$ be defined by the equalities (\ref{equation: m and p}). If $m \geq p + 1$, then:
\begin{flalign*}
F_v(a_1, ..., a_s; m) = m + p, \ \mbox{\cite{LU96},\cite{LRU01}}. && \tag{a} 
\end{flalign*}
\begin{flalign*}
K_{m+p} - C_{2p + 1} = K_{m - p - 1} + \overline{C}_{2p + 1} && \tag{b}
\end{flalign*}
is the only extremal graph in $\mH(a_1, ..., a_s; m)$, \ \cite{LRU01}.
\end{theorem}
The condition $m \geq p + 1$ is necessary according to (\ref{equation: F_v(a_1, ..., a_s; q) exists}). Other proofs of Theorem \ref{theorem: F_v(a_1, ..., a_s; m) = m + p} are given in \cite{Nen00} and \cite{Nen01}.\\

Very little is known about the numbers $F_v(a_1, ..., a_s; m - 1)$. According to (\ref{equation: F_v(a_1, ..., a_s; q) exists}) we have
\begin{equation}
\label{equation: F_v(a_1, ..., a_s; m - 1) exists}
F_v(a_1, ..., a_s; m - 1) \mbox{ exists } \Leftrightarrow m \geq p + 2.
\end{equation}

The following bounds are known:
\begin{equation}
\label{equation: m + p + 2 leq F_v(a_1, ..., a_s; m - 1) leq m + 3p}
m + p + 2 \leq F_v(a_1, ..., a_s; m - 1) \leq m + 3p,
\end{equation}
where the lower bound is true if $p \geq 2$ and the upper bound is true if $p \geq 3$. The lower bound is obtained in \cite{Nen00} and the upper bound is obtained in \cite{KN06a}. In the border case $m = p + 2$ the upper bounds in (\ref{equation: m + p + 2 leq F_v(a_1, ..., a_s; m - 1) leq m + 3p}) are significantly improved in \cite{SXP09}.

When $p = \max\set{a_1, ..., a_s} \leq 5$ we have
\begin{equation}
\label{equation: F_v(a_1, ..., a_s, m - 1) = ...}
F_v(a_1, ..., a_s, m - 1) = \begin{cases}
m + 4, & \emph{if $p = 2$ and $m \geq 6$, \cite{Nen83}}\\
m + 6, & \emph{if $p = 3$ and $m \geq 6$, \cite{Nen02}}\\
m + 7, & \emph{if $p = 4$ and $m \geq 6$, \cite{Nen02}}\\
m + 9, & \emph{if $p = 5$ and $m \geq 7$, \cite{BN15a}}.\\
\end{cases}
\end{equation}
In the cases $p = 2$ and $p = 3$ we also know the numbers: $F_v(2, 2, 2; 3) = 11$, \cite{Myc55} and \cite{Chv79}, $F_v(2, 2, 2, 2; 4) = 11$, \cite{Nen84} (see also \cite{Nen98}), $F_v(2, 2, 3; 4) = 14$, \cite{Nen00} and \cite{CR06}, $F_v(3, 3; 4) = 14$, \cite{Nen81} and \cite{PRU99}. These numbers and the numbers (\ref{equation: F_v(a_1, ..., a_s, m - 1) = ...}) are all the numbers in the form $F_v(a_1, ..., a_s; m - 1)$ when $\max\set{a_1, ..., a_s} \leq 5$. We do not know any of these numbers when $\max\set{a_1, ..., a_s} \geq 6$. In \cite{BN15a} we prove that
\begin{equation}
\label{equation: m + 9 leq F_v(a_1, ..., a_s) leq m + 10}
m + 9 \leq F_v(a_1, ..., a_s; m - 1) \leq m + 10,
\end{equation}
when $\max\set{a_1, ..., a_s} = 6$.

In this paper we complete the computation of these numbers by proving the following
\begin{theorem}
\label{theorem: F_v(a_1, ..., a_s; m - 1) = ..., max set(a_1, ..., a_s) = 6}
Let $a_1, ..., a_s$ be positive integers, such that
\begin{equation*}
2 \leq a_1 \leq ... \leq a_s = 6,
\end{equation*}
and $m = \sum\limits_{i=1}^s (a_i - 1) + 1 \geq 8$. Then
\begin{flalign*}
F_v(a_1, ..., a_s; m - 1) = m + 9, \mbox{ if } a_1 = ... = a_{s - 1} = 2. && \tag{a} 
\end{flalign*}
\begin{flalign*}
F_v(a_1, ..., a_s; m - 1) = m + 10, \mbox{ if } a_{s - 1} \geq 3. && \tag{b} 
\end{flalign*}
\end{theorem}

We obtain the following bounds on numbers of the form $F_v(a_1, ..., a_s; 7)$ where $\max\set{a_1, ..., a_s} = 6$:

\begin{theorem}
\label{theorem: F_v(a_1, ..., a_s; 7) geq F_v(2_(m - 6), 6; 7) geq 3m - 5}
Let $a_1, ..., a_s$ be positive integers such that $\max\set{a_1, ..., a_s} = 6$ and $m = \sum\limits_{i=1}^s (a_i - 1) + 1 \geq 9$. Then
$$F_v(a_1, ..., a_s; 7) \geq F_v(2_{m - 6}, 6; 7) \geq 3m - 5.$$
In particular, $F_v(6, 6; 7) \geq 28$.
\end{theorem}

\begin{theorem}
\label{theorem: F_v(a_1, ..., a_s; 7) leq F_v(6, 6; 7) leq 60}
Let $a_1, ..., a_s$ be positive integers such that $\max\set{a_1, ..., a_s} = 6$ and $m = \sum\limits_{i=1}^s (a_i - 1) + 1$. Then:
\begin{flalign*}
22 \leq F_v(a_1, ..., a_s; 7) \leq F_v(4, 6; 7) \leq 35 \mbox{ if $m = 9$}. && \tag{a} 
\end{flalign*}
\begin{flalign*}
28 \leq F_v(a_1, ..., a_s; 7) \leq F_v(6, 6; 7) \leq 70 \mbox{ if $m = 11$}. && \tag{b} 
\end{flalign*}
\end{theorem}

We also obtain the following results related to the numbers $F_v(a_1, ..., a_s; m - 1)$ where $\max\set{a_1, ..., a_s} = 7$.
\begin{theorem}
\label{theorem: F_v(2, 2, 7; 8) = 20}
$F_v(2, 2, 7; 8) = 20$.
\end{theorem}
\begin{theorem}
\label{theorem: F_v(a_1, ..., a_s; m - 1) leq m + 12, max set(a_1, ..., a_s) = 7}
Let $a_1, ..., a_s$ be positive integers, such that $\max\set{a_1, ..., a_s} = 7$ and $m = \sum\limits_{i=1}^s (a_i - 1) + 1 \geq 9$. Then
\begin{flalign*}
m + 10 \leq F_v(a_1, ..., a_s; m - 1) \leq m + 12. 
\end{flalign*}
\end{theorem}

\begin{remark}
\label{remark: m geq 8}
According to (\ref{equation: F_v(a_1, ..., a_s; m - 1) exists}) the condition $m \geq 8$ from Theorem \ref{theorem: F_v(a_1, ..., a_s; m - 1) = ..., max set(a_1, ..., a_s) = 6} and the condition $m \geq 9$ from Theorem \ref{theorem: F_v(a_1, ..., a_s; m - 1) leq m + 12, max set(a_1, ..., a_s) = 7} are necessary.
\end{remark}

This paper is organized in sections. In the first section the necessary definitions are given, an overview of the known results is provided and at the end we formulate the obtained new results. In the second section we formulate some auxiliary propositions. In the third section we present computer algorithms for finding the maximal graphs in $\mH(a_1, ..., a_s; q; n)$. In the fourth section we find all extremal graphs in $\mH(2, 2, 6; 7)$ and compute the numbers $F_v(2, 2, 6; 7) = 17$ and $F_v(3, 6; 7) = 18$. In the fifth section we prove Theorem \ref{theorem: F_v(a_1, ..., a_s; m - 1) = ..., max set(a_1, ..., a_s) = 6} (a), and in the sixth section we prove Theorem \ref{theorem: F_v(a_1, ..., a_s; m - 1) = ..., max set(a_1, ..., a_s) = 6} (b). In the seventh section we find all graphs in $\mH(2, 2, 6; 7; 18)$ and we prove Theorem \ref{theorem: F_v(a_1, ..., a_s; 7) geq F_v(2_(m - 6), 6; 7) geq 3m - 5} and Theorem \ref{theorem: F_v(a_1, ..., a_s; 7) leq F_v(6, 6; 7) leq 60}. In the eight section we show that $F_v(2, 2, 7; 8) = 20$ (Theorem \ref{theorem: F_v(2, 2, 7; 8) = 20}) and we find some extremal graphs in $\mH(2, 2, 7; 8)$. In the last ninth section we prove Theorem \ref{theorem: F_v(a_1, ..., a_s; m - 1) leq m + 12, max set(a_1, ..., a_s) = 7}.  

This paper has a previous version (\href{https://arxiv.org/abs/1512.02051}{arXiv:1512.02051}). In the current version we use new faster algorithms with the help of which we reproduce the results from the previous version and we also prove new theorems (Theorem \ref{theorem: F_v(a_1, ..., a_s; 7) geq F_v(2_(m - 6), 6; 7) geq 3m - 5}, Theorem \ref{theorem: F_v(a_1, ..., a_s; 7) leq F_v(6, 6; 7) leq 60}, Theorem \ref{theorem: F_v(2, 2, 7; 8) = 20} and Theorem \ref{theorem: F_v(a_1, ..., a_s; m - 1) leq m + 12, max set(a_1, ..., a_s) = 7}).

\section{Some auxiliary results}

Let $a_1, ..., a_s$ be positive integers and $m = \sum\limits_{i=1}^s (a_i - 1) + 1$. Obviously, if $a_i \geq t \geq 2$, then
\begin{equation}
\label{equation: G arrowsv (a_1, ..., a_s) Rightarrow G arrowsv (a_1, ..., a_(i - 1), t, a_i - t, a_(i + 1), ..., a_s)}
G \arrowsv (a_1, ..., a_s) \Rightarrow G \arrowsv (a_1, ..., a_{i - 1}, t, a_i - t + 1, a_{i + 1}, ..., a_s).
\end{equation}
By repeatedly applying (\ref{equation: G arrowsv (a_1, ..., a_s) Rightarrow G arrowsv (a_1, ..., a_(i - 1), t, a_i - t, a_(i + 1), ..., a_s)}) we obtain
\begin{equation}
\label{equation: G arrowsv (a_1, ..., a_s) Rightarrow G arrowsv (2_(m - 1))}
G \arrowsv (a_1, ..., a_s) \Rightarrow G \arrowsv (2_{m - 1}).
\end{equation}
Since $G \arrowsv (2_{m - 1}) \Leftrightarrow \chi(G) \geq m$, from (\ref{equation: G arrowsv (a_1, ..., a_s) Rightarrow G arrowsv (2_(m - 1))}) it follows
\begin{equation}
\label{equation: G arrowsv (a_1, ..., a_s) Rightarrow chi(G) geq m}
G \arrowsv (a_1, ..., a_s) \Rightarrow \chi(G) \geq m. \cite{Nen01}
\end{equation}
This fact will be necessary in the proof of Theorem \ref{theorem: rpp}. Very simple examples of graphs for which equality is reached in (\ref{equation: G arrowsv (a_1, ..., a_s) Rightarrow chi(G) geq m}) are obtained in \cite{Nen85}.

\begin{conjecture}
\label{conjecture: chi(G) leq m + 1}
If $G \in \mH_{extr}(a_1, ..., a_s; m - 1)$, then $\chi(G) \leq m + 1$.
\end{conjecture}
For all known examples of extremal graphs, including the extremal graphs obtained in this paper, this inequality holds.\\

Let the numbers $a_1, ..., a_s$ satisfy the inequalities
\begin{equation*}
2 \leq a_1 \leq ... \leq a_s = p.
\end{equation*} 
As before, by repeatedly applying (\ref{equation: G arrowsv (a_1, ..., a_s) Rightarrow G arrowsv (a_1, ..., a_(i - 1), t, a_i - t, a_(i + 1), ..., a_s)}) we obtain that if $a_{s - 1} \geq 3$
\begin{equation}
\label{equation: G arrowsv (a_1, ..., a_s) Rightarrow G arrowsv (2_(m - p - 2), 3, p)}
G \arrowsv (a_1, ..., a_s) \Rightarrow G \arrowsv (2_{m - p - 2}, 3, p),
\end{equation}
and therefore it is true that
\begin{equation}
\label{equation: F_v(2_(m - p - 2), 3, p; m - 1) leq F_v(a_1, ..., a_s; m - 1)}
F_v(2_{m - p - 2}, 3, p; m - 1) \leq F_v(a_1, ..., a_s; m - 1).
\end{equation}

Further, we will use the following obvious
\begin{proposition}
\label{proposition: G - A arrowsv (a_1, ..., a_(i - 1), a_i - 1, a_(i + 1_, ..., a_s)}
Let $G$ be a graph, $G \arrowsv (a_1, ..., a_s)$ and $A \subseteq \V(G)$ be an independent set. Then if $a_i \geq 2$
\begin{equation*}
G - A \arrowsv (a_1, ..., a_{i - 1}, a_i - 1, a_{i + 1}, ..., a_s).
\end{equation*}
\end{proposition}
Let $G \in \mH(a_1, ..., a_s; m - 1; n)$ and $A$ be an independent set of vertices of $G$ such that $\abs{A}$ = $\alpha(G)$. According to Proposition \ref{proposition: G - A arrowsv (a_1, ..., a_(i - 1), a_i - 1, a_(i + 1_, ..., a_s)}, $G - A \in \mH(a_1 - 1, ..., a_s; m - 1; n - \abs{A})$ and by Theorem \ref{theorem: F_v(a_1, ..., a_s; m) = m + p}, $n - \abs{A} \geq m - 1 + p$. Thus, we proved
\begin{equation}
\label{equation: G in mH(a_1, ..., a_s; m - 1; n) Rightarrow alpha(G) leq n - m - p + 1}
G \in \mH(a_1, ..., a_s; m - 1; n) \Rightarrow \alpha(G) \leq n - m - p + 1.
\end{equation}

We will also need the following improvement of the lower bound in (\ref{equation: m + p + 2 leq F_v(a_1, ..., a_s; m - 1) leq m + 3p})
\begin{theorem}
\label{theorem: F_v(a_1, ..., a_s; m - 1) geq m + p + 3}
\cite{Nen02}
Let $a_1, ..., a_s$ be positive integers, let $m$ and $p$ be defined by the equalities (\ref{equation: m and p}), $p \geq 3$ and $m \geq p + 2$. If $F_v(2, 2, p; p + 1) \geq 2p + 5$, then
\begin{equation*}
F_v(a_1, ..., a_s; m - 1) \geq m + p + 3.
\end{equation*}
\end{theorem}

\section{Algorithms}

In this section we present algorithms for finding all graphs in $\mH_{max}(a_1, ..., a_s; q; n)$ with the help of a computer. The remaining graphs in $\mH(a_1, ..., a_s; q; n)$ can be obtained by removing edges from the maximal graphs. The algorithms are modifications of the algorithm from \cite{BN16}. However, we will present them in detail since they are essential to this paper. The idea for these algorithms comes from \cite{PRU99} (see Algorithm A1). Similar algorithms are used in \cite{CR06}, \cite{XLS10}, \cite{LR11}, \cite{SLPX12}, \cite{BN15a} and \cite{BN15b}. Also with the help of a computer, results for Folkman numbers are obtained in \cite{JR95}, \cite{SXP09}, \cite{SXL09} and \cite{DLSX13}. Let us also note the important role of the \emph{nauty} programs \cite{MP13} in this work. We use them for fast generation of non-isomorphic graphs and graph isomorph rejection.

Let $2 \leq a_1 \leq ... \leq a_s = p$ be positive integers, $m = \sum\limits_{i=1}^s (a_i - 1) + 1$, and let $G \in \mH(a_1, ..., a_s; m - 1; n), n \geq m - 1$. It is clear that $\alpha(G) \geq 2$, and according to (\ref{equation: G in mH(a_1, ..., a_s; m - 1; n) Rightarrow alpha(G) leq n - m - p + 1}), $\alpha(G) \leq n - m - p + 1$. As we will see further in the proofs of the results of this paper, it is computationally most difficult to obtain the graphs $G$ for which $\alpha(G) = 2$. Therefore, first we present in detail the Algorithm \ref{algorithm: mH_(max)(a_1, ..., a_s; q; n), alpha(G) = 2} for finding all graphs $G \in \mH_{max}(a_1, ..., a_s; m - 1; n)$ with $\alpha(G) = 2$, even though it is a special case of the more general Algorithm \ref{algorithm: mH_(max)(a_1, ..., a_s; q; n)}. Algorithm \ref{algorithm: mH_(max)(a_1, ..., a_s; q; n), alpha(G) = 2} is based on the following propositions, which are easy to prove:

\begin{proposition}
\label{proposition: H is a (+K_(q - 1))-graph}
Let $G \in \mH_{max}(a_1, ..., a_s; q; n)$, $A$ be an independent set of vertices of $G$ and $H = G - A$. Then $H \in \mH_{+K_{q - 1}}(a_1 - 1, ..., a_s; q; n - \abs{A})$.
\end{proposition}

\begin{proof}
According to Proposition \ref{proposition: G - A arrowsv (a_1, ..., a_(i - 1), a_i - 1, a_(i + 1_, ..., a_s)}, $G - A \in \mH(a_1 - 1, ..., a_s; q; n - \abs{A})$. From the maximality of $G$ it follows that $G - A$ is a $(+K_{q - 1})$-graph.
\end{proof}

\begin{proposition}
\label{proposition: maximal K_q-free graph G conditions}
Let $G$ be a maximal $K_q$-free graph and $v_1, v_2$ be non-adjacent vertices of $G$. Then
	
$K_{q - 2} \subseteq N_G(v_1) \cap N_G(v_2)$.
\end{proposition}

\begin{proposition}
\label{proposition: alpha(G) = 2 conditions}
Let $G$ be a graph,  $v_1, v_2$ be non-adjacent vertices of $G$ and $H = G - \set{v_1, v_2}$. Then $\alpha(G) = 2$ if and only if the following three conditions are satisfied:

(a) $\alpha(H) \leq 2$.

(b) $\alpha(H - N_G(v_j)) \leq 1, \ j = 1, 2$, i.e. either $N_G(v_j) = \V(H)$ or $H - N_G(v_j)$ is a clique.
	
(c) $N_G(v_1) \cup N_G(v_2) = \V(H)$.
\end{proposition}
Further, we will prove the more general Proposition \ref{proposition: alpha(G) leq t Leftrightarrow alpha(H - bigcup_(v in A') N_G((v)) leq t - abs(A')}.

Now, we formulate the first of the two important algorithms in this paper which finds all graphs $G \in \mH_{max}(a_1, ..., a_s; q; n)$ with $\alpha(G) = 2$.

\begin{algorithm}
	\label{algorithm: mH_(max)(a_1, ..., a_s; q; n), alpha(G) = 2} The input of the algorithm is the set $\mathcal{A}$ of all graphs in $\mH_{max}(a_1 - 1, ..., a_s; q; n - 2)$ with independence number not greater than 2. The output of the algorithm is the set $\mathcal{B}$ of all graphs $G \in \mH_{max}(a_1, ..., a_s; q; n)$ with $\alpha(G) = 2$.
	
	\emph{1.} By removing edges from the graphs in $\mathcal{A}$ obtain the set
	
	$\mathcal{A}' = \set{H \in \mH_{+K_{q - 1}}(a_1 - 1, ..., a_s; q; n - 2) : \alpha(H) \leq 2}$.
	
	\emph{2.} For each graph $H \in \mathcal{A}'$:
	
	\emph{2.1.} Find the family $\mathcal{M}(H) = \set{M_1, ..., M_t}$ of all maximal $K_{q - 1}$-free subsets of $\V(H)$.
	
	\emph{2.2.} Find all pairs $N = \set{M_{i_1}, M_{i_2}}$ of elements of $\mathcal{M}(H)$ (it is possible that $M_{i_1} = M_{i_2}$), which fulfill the conditions:
	
	(a) $K_{q - 2} \subseteq M_{i_1} \cap M_{i_2}$.
	
	(b) $\alpha(H - M_{i_j}) \leq 1, \ j = 1, 2$.

	(c) $M_{i_1} \cup M_{i_2} = \V(H)$.
	
	\emph{2.3.} For each pair $N = \set{M_{i_1}, M_{i_2}}$ of elements of $\mathcal{M}(H)$ found in step 2.2 construct the graph $G = G(N)$ by adding new non-adjacent vertices $v_1, v_2$ to $\V(H)$ such that $N_G(v_j) = M_{i_j}, j = 1, 2$. If $\omega(G + e) = q, \forall e \in \E(\overline{G})$, then add $G$ to $\mathcal{B}$.
	
	\emph{3.} Remove the isomorph copies of graphs from $\mathcal{B}$.
	
	\emph{4.} Remove from the obtained in step 3 set $\mathcal{B}$ all graphs $G$ for which $G \not\arrowsv (a_1, ..., a_s)$.
\end{algorithm}

\begin{theorem}
	\label{theorem: algorithm mH_(max)(a_1, ..., a_s; q; n), alpha(G) = 2}
	\cite{BN16}
	After the execution of Algorithm \ref{algorithm: mH_(max)(a_1, ..., a_s; q; n), alpha(G) = 2}, the obtained set $\mathcal{B}$ coincides with the set of all graphs $G \in \mH_{max}(a_1, ..., a_s; q; n)$ with $\alpha(G) = 2$.
\end{theorem}

\begin{proof}
Suppose that after the execution of Algorithm \ref{algorithm: mH_(max)(a_1, ..., a_s; q; n), alpha(G) = 2} the graph $G \in \mathcal{B}$. Then, $G = G(N)$ where $N$ and the following notations are the same as in step 2.3. Since $H = G - \set{v_1, v_2} \in \mathcal{A'}$, we have $\omega(H) < q$. Since $N_G(v_1)$ and $N_G(v_2)$ are $K_{q - 1}$-free sets, it follows that $\omega(G) < q$. The check at the end of step 2.3 guarantees that $G$ is a maximal $K_q$-free graph and the check in step 4 guarantees that $G \arrowsv (a_1, ..., a_s)$, therefore $G \in \mH_{max}(a_1, ..., a_s; q; n)$. Again, by $H \in \mathcal{A'}$, we have $\alpha(H) \leq 2$ and from the conditions (b) and (c) in step 2.2 and Proposition \ref{proposition: alpha(G) = 2 conditions} it follows that $\alpha(G) = 2$.

Let $G \in \mH_{max}(a_1, ..., a_s; q; n)$ and $\alpha(G) = 2$. We will prove that, after the execution of Algorithm \ref{algorithm: mH_(max)(a_1, ..., a_s; q; n), alpha(G) = 2}, $G \in \mathcal{B}$. Let $v_1, v_2$ be non-adjacent vertices in $G$ and $H = G - \set{v_1, v_2}$. Then $\alpha(H) \leq 2$ and according to Proposition \ref{proposition: H is a (+K_(q - 1))-graph}, $H \in \mathcal{A}'$. Since $G$ is a maximal $K_q$-free graph, $N_G(v_1)$ and $N_G(v_2)$ are maximal $K_{q - 1}$-free subsets of $V(H)$, and therefore $N_G(v_i) \in \mathcal{M}(H), i = 1, 2$ (see step 2.1). Let $N = \set{N_G(v_1), N_G(v_2)}$. By Proposition \ref{proposition: maximal K_q-free graph G conditions}, $N$ fulfills the condition (a), and by Proposition \ref{proposition: alpha(G) = 2 conditions}, $N$ also fulfills (b) and (c). Thus, we showed that $N$ fulfills all conditions in step 2.2, and since $G = G(N)$ is a maximal $K_q$-free graph, in step 2.3 $G$ is added to $\mathcal{B}$. Clearly, after step 4 the graph $G$ remains in $\mathcal{B}$.
\end{proof}

We shall now generalize Algorithm \ref{algorithm: mH_(max)(a_1, ..., a_s; q; n), alpha(G) = 2} to find all graphs $G \in \mH_{max}(a_1, ..., a_s; m - 1; n)$ for which $r \leq \alpha(G) \leq t$. We will need the following proposition, which is a generalization of Proposition \ref{proposition: alpha(G) = 2 conditions} (in the special case $t = 2$, when $G$ is not a complete graph Proposition \ref{proposition: alpha(G) leq t Leftrightarrow alpha(H - bigcup_(v in A') N_G((v)) leq t - abs(A')} coincides with Proposition \ref{proposition: alpha(G) = 2 conditions}). 

\begin{proposition}
	\label{proposition: alpha(G) leq t Leftrightarrow alpha(H - bigcup_(v in A') N_G((v)) leq t - abs(A')}
	Let $A$ be an independent set of vertices of $G$ and $H = G - A$. Then, 
	
	$\alpha(G) \leq t \Leftrightarrow \alpha(H - \bigcup_{v \in A'} N_G(v)) \leq t - \abs{A'}, \ \forall A' \subseteq A$.
\end{proposition}

\begin{proof}
	Let $\alpha(G) \leq t$. Suppose that for some $A' \subseteq A$ we have $\alpha(H - \bigcup_{v \in A'} N_G(v)) > t - \abs{A'}$. Consequently, there exists an independent set $A''$ of vertices of $H - \bigcup_{v \in A'} N_G(v)$ such that $\abs{A''} > t - \abs{A'}$. We obtained that the independent set $A' \cup A''$ has more than $t$ vertices, which is a contradiction.
	
	Now, let $\alpha(H - \bigcup_{v \in A'} N_G(v)) \leq t - \abs{A'}, \ \forall A' \subseteq A$. Let $\widetilde{A}$ be an independent set of vertices of $G$ and $\abs{\widetilde{A}} = \alpha(G)$. Then, $\widetilde{A} = A_1 \cup A_2$ where $A_1 \subseteq A$ and $A_2$ is an independent set in $H - \bigcup_{v \in A_1} N_G(v)$. Since $\abs{A_2} \leq \alpha(H - \bigcup_{v \in A_1} N_G(v)) \leq t - \abs{A_1}$, we obtain $\alpha(G) = \abs{\widetilde{A}} = \abs{A_1} + \abs{A_2} \leq t$. 
\end{proof}

Now we move on to the formulation of the second important Algorithm \ref{algorithm: mH_(max)(a_1, ..., a_s; q; n)}, which is a generalization of Algorithm \ref{algorithm: mH_(max)(a_1, ..., a_s; q; n), alpha(G) = 2} and finds all graphs $G \in \mH_{max}(a_1, ..., a_s; q; n)$ with $r \leq \alpha(G) \leq t$.

\begin{algorithm}
	\label{algorithm: mH_(max)(a_1, ..., a_s; q; n)} The input of the algorithm is the set $\mathcal{A}$ of all graphs in $\mH_{max}(a_1 - 1, ..., a_s; q; n - r)$ with independence number not greater than t. The output of the algorithm is a set $\mathcal{B}$ of all graphs $G \in \mH_{max}(a_1, ..., a_s; q; n)$ with $r \leq \alpha(G) \leq t$.
	
	\emph{1.} By removing edges from the graphs in $\mathcal{A}$ obtain the set
	
	$\mathcal{A}' = \set{H \in \mH_{+K_{q - 1}}(a_1 - 1, ..., a_s; q; n - r) : \alpha(H) \leq t}$.
	
	\emph{2.} For each graph $H \in \mathcal{A}'$:
	
	\emph{2.1.} Find the family $\mathcal{M}(H) = \set{M_1, ..., M_l}$ of all maximal $K_{q - 1}$-free subsets of $\V(H)$.
	
	\emph{2.2.} Find all $r$-element multisets $N = \set{M_{i_1}, M_{i_2}, ..., M_{i_r}}$ of elements of $\mathcal{M}(H)$, which fulfill the conditions:
	
	(a) $K_{q - 2} \subseteq M_{i_j} \cap M_{i_k}$ for every $M_{i_j}, M_{i_k} \in N$.
	
	(b) $\alpha(H - \bigcup_{M_{i_j} \in N'} M_{i_j}) \leq t - \abs{N'}$ for subtuple $N'$ of $N$.
	
	\emph{2.3.} For each $r$-element multiset $N = \set{M_{i_1}, M_{i_2}, ..., M_{i_r}}$ of elements of $\mathcal{M}(H)$ found in step 2.2 construct the graph $G = G(N)$ by adding new independent vertices $v_1, v_2, ..., v_r$ to $\V(H)$ such that $N_G(v_j) = M_{i_j}, j = 1, ..., r$. If $\omega(G + e) = q, \forall e \in \E(\overline{G})$, then add $G$ to $\mathcal{B}$.
	
	\emph{3.} Remove the isomorph copies of graphs from $\mathcal{B}$.

	\emph{4.} Remove from the obtained in step 3 set $\mathcal{B}$ all graphs $G$ for which $G \not\arrowsv (a_1, ..., a_s)$.
		
\end{algorithm}

\begin{theorem}
	\label{theorem: algorithm mH_(max)(a_1, ..., a_s; q; n)}
	\cite{BN16}
	After the execution of Algorithm \ref{algorithm: mH_(max)(a_1, ..., a_s; q; n), alpha(G) = 2}, the obtained set $\mathcal{B}$ coincides with the set of all graphs $G \in \mH_{max}(a_1, ..., a_s; q; n)$ with $r \leq \alpha(G) \leq t$.
\end{theorem}

\begin{proof}
Suppose that after the execution of Algorithm \ref{algorithm: mH_(max)(a_1, ..., a_s; q; n)}, $G \in \mathcal{B}$. In the same way as in the proof of Theorem \ref{theorem: algorithm mH_(max)(a_1, ..., a_s; q; n), alpha(G) = 2}, we prove that $G \in \mH_{max}(a_1, ..., a_s; q)$. From Proposition \ref{proposition: alpha(G) leq t Leftrightarrow alpha(H - bigcup_(v in A') N_G((v)) leq t - abs(A')} and the condition (b) in step 2.2 it follows that $\alpha(G) \leq t$, and step 2.3 guaranties that $\alpha(G) \geq r$.

Now let $G \in \mH_{max}(a_1, ..., a_s; q; n)$, $r \leq \alpha(G) \leq t$, let $A$ be an independent set of $r$ vertices of $G$ and $H = G - A$. Then obviously, $\alpha(H) \leq t$ and according to Proposition \ref{proposition: H is a (+K_(q - 1))-graph}, $H \in \mathcal{A}'$ (see step 1). By repeating the reasoning of the second part of the proof of Theorem \ref{theorem: algorithm mH_(max)(a_1, ..., a_s; q; n), alpha(G) = 2}, we prove that after the execution of Algorithm \ref{algorithm: mH_(max)(a_1, ..., a_s; q; n)}, $G \in \mathcal{B}$.
\end{proof}

At the end of this section, we will propose a method to improve Algorithm \ref{algorithm: mH_(max)(a_1, ..., a_s; q; n), alpha(G) = 2} and Algorithm \ref{algorithm: mH_(max)(a_1, ..., a_s; q; n)} which is based on the following proposition which is easy to prove:
\begin{proposition}
Let $G \in \mH(2, 2, p; p + 1)$ and $v \in \V(G)$. Then, all non-neighbors of $v$ induce a graph with chromatic number greater than 2. In particular, from $G \in \mH(2, 2, p; p + 1)$ it follows that $\Delta(G) \leq \abs{\V(G)} - 4$.
\end{proposition}
As we will see further (see Table \ref{table: H(2, 2, 6; 7; 17) properties} and Table \ref{table: H(2, 2, 6; 7; 18) properties}), the inequality $\Delta(G) \leq \abs{\V(G)} - 4$ is exact. In some special cases, for example the proofs of Theorem \ref{theorem: abs(mathcal(H)(2, 2, 6; 7; 17)) = 3}, Theorem \ref{theorem: abs(mathcal(H)(2, 2, 6; 7; 18)) = 76515} and Theorem \ref{theorem: F_v(2, 2, 7; 8) = 20}, we can use the inequality $\Delta(G) \leq \abs{\V(G)} - 4$ to speed up computations in some parts of the proofs. We used this inequality only to make sure that the obtained results are correct.\\

All computations were done on a personal computer. The proofs of Theorem \ref{theorem: abs(mathcal(H)(2, 2, 6; 7; 18)) = 76515} and Theorem \ref{theorem: F_v(2, 2, 7; 8) = 20} were the most time consuming, each taking about a month to complete.

\section{Finding all graphs in $\mH(2, 2, 6; 7; 17)$ and computation of the numbers $F_v(2, 2, 6; 7)$ and $F_v(3, 6; 7)$}

Let $a_1, ..., a_s$ be positive integers and let $m$ and $p$ be defined by (\ref{equation: m and p}). According to (\ref{equation: F_v(a_1, ..., a_s; m - 1) exists}), $F_v(a_1, ..., a_s; m - 1)$ exists if and only if $m \geq p + 2$. In the border case $m = p + 2, \ p \geq 3$, there are only two canonical numbers in the form $F_v(a_1, ..., a_s; m - 1)$, namely $F_v(2, 2, p; p + 1)$ and $F_v(3, p; p + 1)$. The computation of the numbers $F_v(a_1, ..., a_s; m - 1)$ when $\max\set{a_1, ..., a_s} = 6$, i.e. the proof of Theorem \ref{theorem: F_v(a_1, ..., a_s; m - 1) = ..., max set(a_1, ..., a_s) = 6}, will be done with the help of the numbers $F_v(2, 2, 6; 7)$ and $F_v(3, 6; 7)$. Because of this, we will first compute these two numbers by proving
\begin{theorem}
\label{theorem: F_v(2, 2, 6; 7) = 17 and F_v(3, 6; 7) = 18}
$F_v(2, 2, 6; 7) = 17$ and $F_v(3, 6; 7) = 18$.
\end{theorem}
From (\ref{equation: G arrowsv (a_1, ..., a_s) Rightarrow G arrowsv (a_1, ..., a_(i - 1), t, a_i - t, a_(i + 1), ..., a_s)}) it is easy to see that
\begin{equation*}
G \arrowsv (3, p) \Rightarrow G \arrowsv (2, 2, p)
\end{equation*}
and therefore $F_v(2, 2, p; p + 1) \leq F_v(3, p; p + 1)$. In \cite{KN06c} the following problem is formulated:
\begin{problem}
\label{problem: F_v(2, 2, p; p + 1) neq F_v(3, p; p + 1)}
\cite{KN06c}
Does there exist a positive integer $p$ for which $F_v(2, 2, p; p + 1) \neq F_v(3, p; p + 1)$?
\end{problem}
Theorem \ref{theorem: F_v(2, 2, 6; 7) = 17 and F_v(3, 6; 7) = 18} gives a positive answer to Problem \ref{problem: F_v(2, 2, p; p + 1) neq F_v(3, p; p + 1)}. Since
\begin{equation*}
F_v(2, 2, p; p + 1) = F_v(3, p; p + 1), \ p \leq 5
\end{equation*}
(see \cite{BN15a}), it becomes clear that $p = 6$ is the smallest positive integer for which
\begin{equation*}
F_v(2, 2, p; p + 1) \neq F_v(3, p; p + 1)
\end{equation*}

For the proof Theorem \ref{theorem: F_v(2, 2, 6; 7) = 17 and F_v(3, 6; 7) = 18} we will need the following
\begin{theorem}
\label{theorem: abs(mathcal(H)(2, 2, 6; 7; 17)) = 3}
$\abs{\mH(2, 2, 6; 7; 17)} = 3$ and $\mH_{extr}(2, 2, 6; 7) = \mH(2, 2, 6; 7; 17) = \set{G_1, G_2, G_3}$ (see Figure \ref{figure: H(2, 2, 6; 7; 17)}).
\end{theorem}

\begin{proof}
We will find all graphs in $\mH(2, 2, 6; 7; 17)$ with the help of a computer. Let $G \in \mH(2, 2, 6; 7; 17)$. Clearly $\alpha(G) \geq 2$, and according to (\ref{equation: G in mH(a_1, ..., a_s; m - 1; n) Rightarrow alpha(G) leq n - m - p + 1}), $\alpha(G) \leq 4$.

First, we prove that there are no graphs in $\mH_{max}(2, 2, 6; 7; 17)$ with independence number 4. According to Theorem \ref{theorem: F_v(a_1, ..., a_s; m) = m + p}, $\overline{C}_{13}$ is the only graph in $\mH(2, 6; 7; 13)$. Starting from $\mH_{max}(2, 6; 7; 13) = \set{\overline{C}_{13}}$, by applying Algorithm \ref{algorithm: mH_(max)(a_1, ..., a_s; q; n)} ($r = 4; t = 4$) we do not obtain any graphs, and from Theorem \ref{theorem: algorithm mH_(max)(a_1, ..., a_s; q; n)} it follows that there are no graphs in $\mH_{max}(2, 2, 6; 7; 17)$ with independence number 4.

Now, we shall prove that there are no graphs in $\mH_{max}(2, 2, 6; 7; 17)$ with independence number 3. It is clear that $K_6$ is the only graph in $\mH_{max}(3; 7; 6)$. Starting from $\mH_{max}(3; 7; 6) = \set{K_6}$ by applying Algorithm \ref{algorithm: mH_(max)(a_1, ..., a_s; q; n)} ($r = 2; t = 3$) we obtain all graphs with independence number not greater than 3 in $\mH_{max}(4; 7; 8)$. In the same way, we successively obtain all graphs with independence number not greater than 3 in $\mH_{max}(5; 7; 10)$, $\mH_{max}(6; 7; 12)$, $\mH_{max}(2, 6; 7; 14)$. In the end, no graphs are produced by applying Algorithm \ref{algorithm: mH_(max)(a_1, ..., a_s; q; n)} ($r = 3; t = 3$) to the obtained graphs in $\mH_{max}(2, 6; 7; 14)$ with independence number not greater than 3, and from Theorem \ref{theorem: algorithm mH_(max)(a_1, ..., a_s; q; n)} we conclude that there are no graphs in $\mH_{max}(2, 2, 6; 7; 17)$ with independence number 3.

The last part of the proof is to find all graphs in $\mH_{max}(2, 2, 6; 7; 17)$ with independence number 2. It is clear that $K_7 - e$ is the only graph in $\mH_{max}(3; 7; 7)$. By applying Algorithm \ref{algorithm: mH_(max)(a_1, ..., a_s; q; n), alpha(G) = 2} we successively obtain all graphs with independence number 2 in $\mH_{max}(4; 7; 9)$, $\mH_{max}(5; 7; 11)$, $\mH_{max}(6; 7; 13)$, $\mH_{max}(2, 6; 7; 15)$ and $\mH_{max}(2, 2, 6; 7; 17)$. As a result, we obtain the graph $G_1 \in \mH_{max}(2, 2, 6; 7; 17)$, which is shown on Figure \ref{figure: H(2, 2, 6; 7; 17)}. According to Theorem \ref{theorem: algorithm mH_(max)(a_1, ..., a_s; q; n), alpha(G) = 2}, $G_1$ is the only graph in $\mH_{max}(2, 2, 6; 7; 17)$ with independence number 2. Since there are no graphs in $\mH_{max}(2, 2, 6; 7; 17)$ with independence number greater than 2, we proved that $\mH_{max}(2, 2, 6; 7; 17) = \set{G_1}$

The number of maximal $K_7$-free graphs and $(+K_6)$-graphs obtained in each step of the proof is shown in Table \ref{table: finding all graphs in H(2, 2, 6; 7; 17)}. By removing edges from $G_1$ we find that there are only two other graphs in $\mH(2, 2, 6; 7; 17)$, which we will denote by $G_2$ and $G_3$ (see Figure \ref{figure: H(2, 2, 6; 7; 17)}). Let us also note, that $G_1 \supset G_2 \supset G_3$ and for the graphs $G_1$, $G_2$ and $G_3$ the inequality (\ref{equation: G arrowsv (a_1, ..., a_s) Rightarrow chi(G) geq m}) is strict (see Conjecture \ref{conjecture: chi(G) leq m + 1}). It is clear that $G_3$ is the only minimal graph in $\mH(2, 2, 6; 7; 17)$. Some properties of the graphs $G_1$, $G_2$ and $G_3$ are given in Table \ref{table: H(2, 2, 6; 7; 17) properties}.
From (\ref{equation: m + 9 leq F_v(a_1, ..., a_s) leq m + 10}) ($m = 8, p = 6$) we obtain
\begin{equation}
\label{equation: 17 leq F_v(2, 2, 6; 7) leq F_v(3, 6; 7) leq 18}
17 \leq F_v(2, 2, 6; 7) \leq F_v(3, 6; 7) \leq 18.
\end{equation}
The inequality $F_v(2, 2, 6; 7) \geq 17$ also follows from the fact, that the graphs $G_1$, $G_2$ and $G_3$ have no isolated vertices. The inequality $F_v(3, 6; 7) \leq 18$ was first proved in \cite{SXP09}.

From (\ref{equation: 17 leq F_v(2, 2, 6; 7) leq F_v(3, 6; 7) leq 18}) it follows $\mH_{extr}(2, 2, 6; 7) = \mH(2, 2, 6; 7; 17) = \set{G_1, G_2, G_3}$. Thus, we finish the proof of Theorem \ref{theorem: abs(mathcal(H)(2, 2, 6; 7; 17)) = 3}.\\
\end{proof}

\begin{table}[h]
	\centering
	\begin{tabular}{ | l | r | r | r | r | r | r | }
		\hline
		Graph		& $|\E(G)|$		& $\delta(G)$	& $\Delta(G)$	& $\alpha(G)$	& $\chi(G)$		& $|Aut(G)|$	\\
		\hline
		$G_{1}$		&  108			& 12			& 13			& 2				& 9				& 2		\\
		$G_{2}$		&  107			& 11			& 13			& 2				& 9				& 4		\\
		$G_{3}$		&  106			& 11			& 13			& 2				& 9				& 40	\\
		\hline
	\end{tabular}
	\caption{The graphs in $\mH(2, 2, 6; 7; 17)$ and some of their properties}
	\label{table: H(2, 2, 6; 7; 17) properties}
\end{table}

\begin{figure}
	\centering
	\begin{subfigure}{0.5\textwidth}
		\centering
		\includegraphics[height=224px,width=112px]{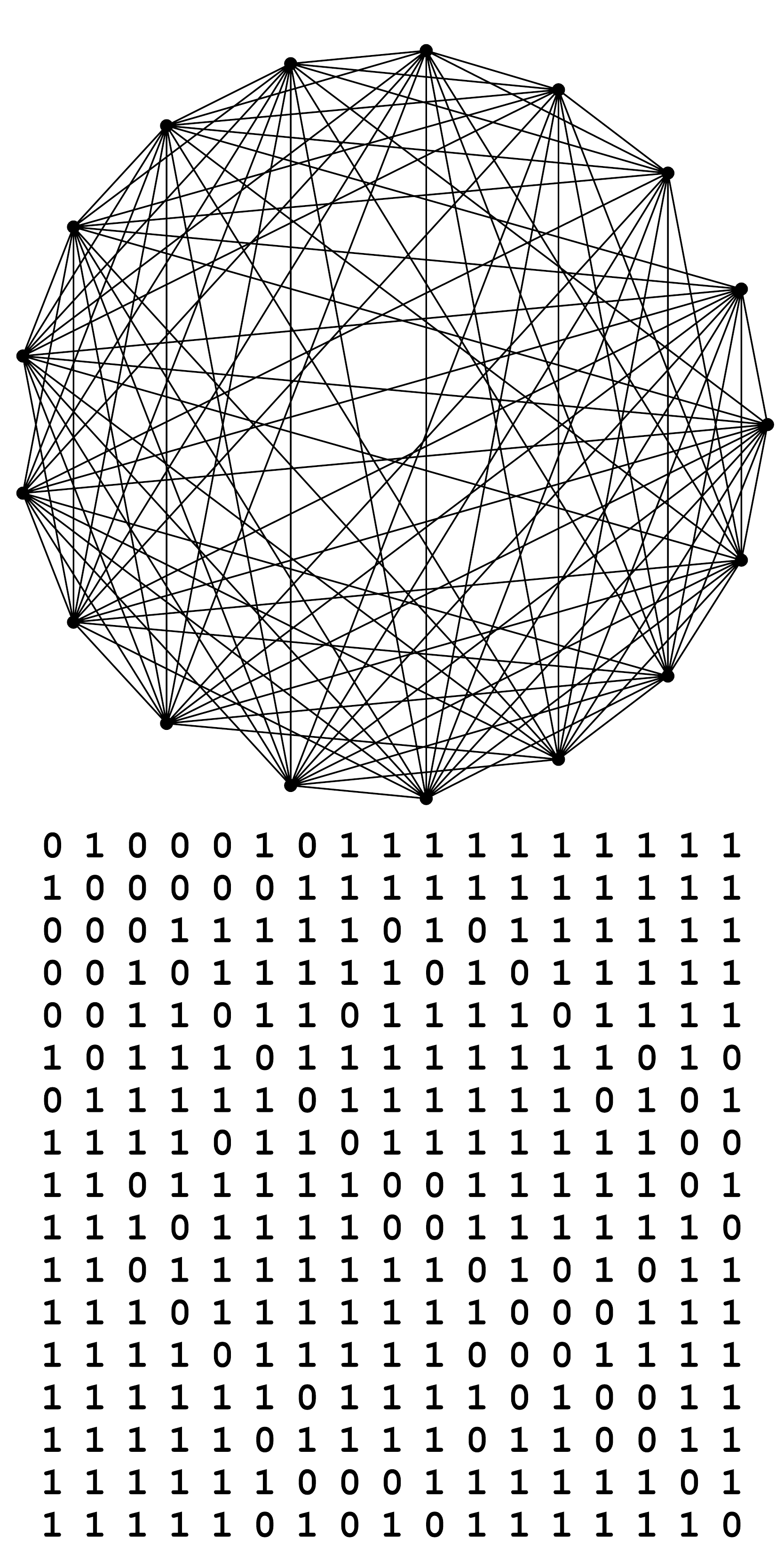}
		\caption*{\emph{$G_{1}$}}
		\label{figure: G_1}
	\end{subfigure}%
	\begin{subfigure}{0.5\textwidth}
		\centering
		\includegraphics[height=224px,width=112px]{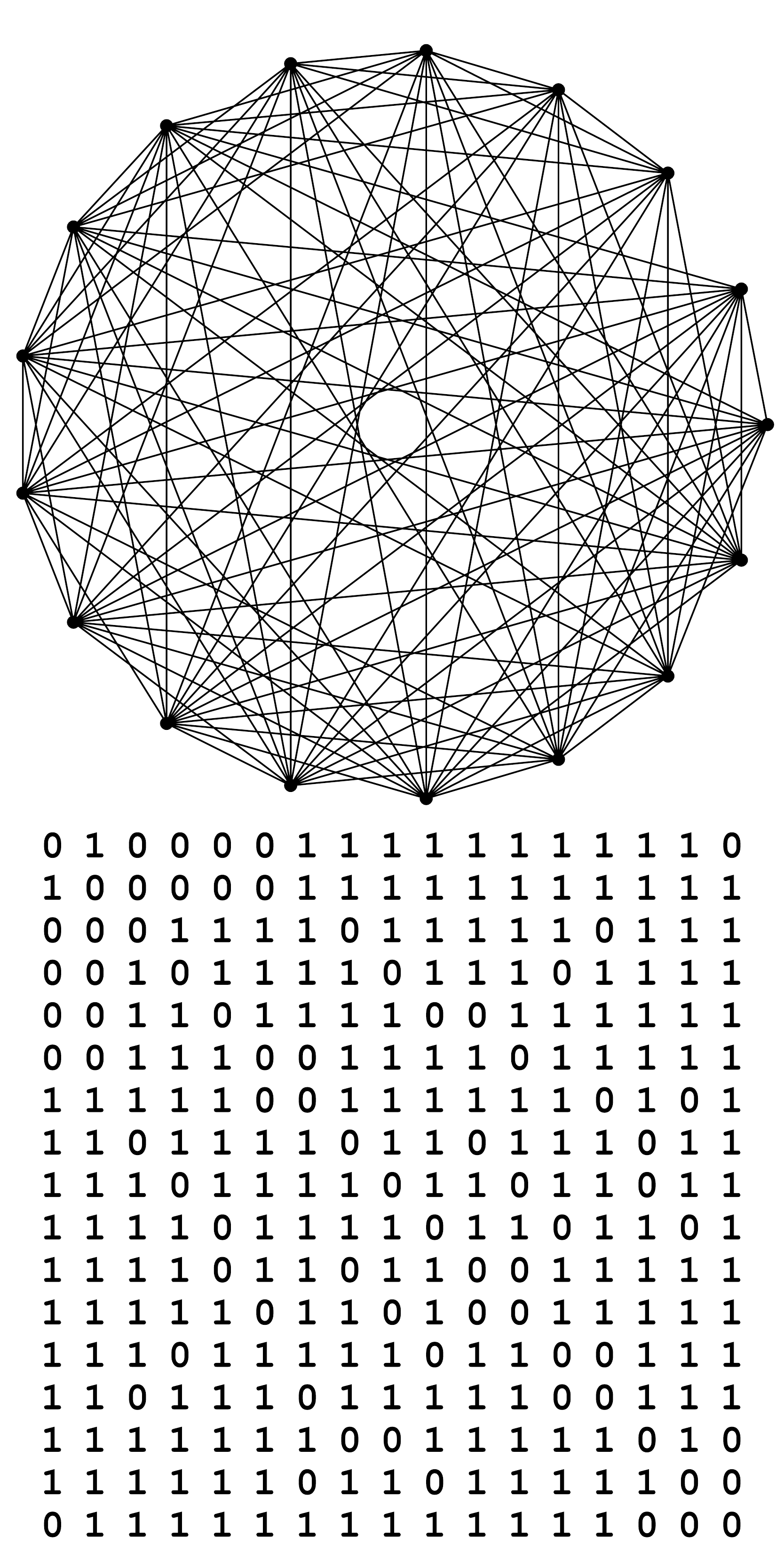}
		\caption*{\emph{$G_{2}$}}
		\label{figure: G_2}
	\end{subfigure}
	
	\begin{subfigure}{0.5\textwidth}
		\centering
		\includegraphics[height=224px,width=112px]{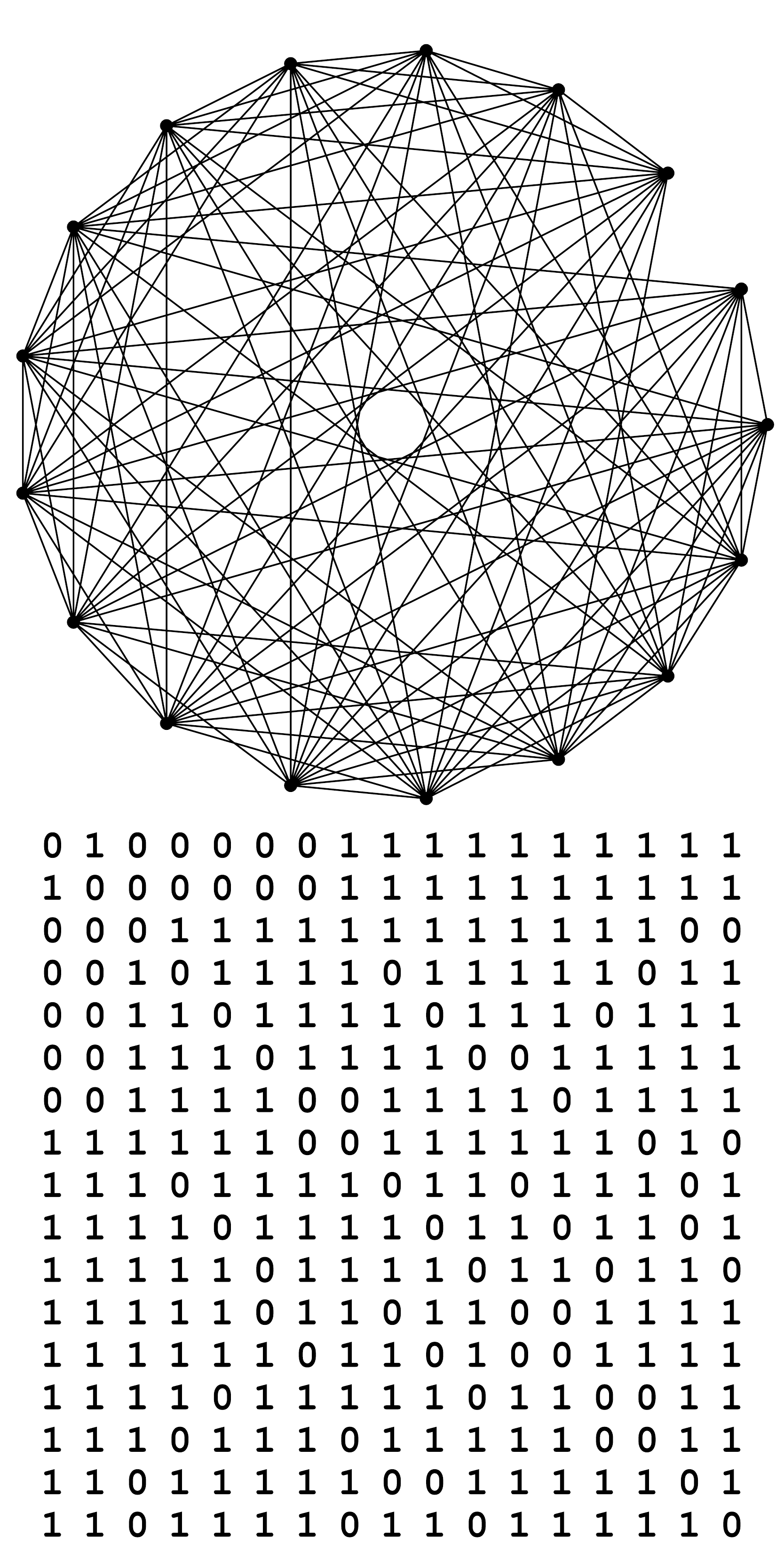}
		\caption*{\emph{$G_{3}$}}
		\label{figure: G_3}
	\end{subfigure}%
	\caption{All 3 graphs in $\mH(2, 2, 6; 7; 17)$}
	\label{figure: H(2, 2, 6; 7; 17)}
\end{figure}

\subsection*{Proof of Theorem \ref{theorem: F_v(2, 2, 6; 7) = 17 and F_v(3, 6; 7) = 18}}

The equality $F_v(2, 2, 6; 7) = 17$ follows from Theorem \ref{theorem: abs(mathcal(H)(2, 2, 6; 7; 17)) = 3}. According to (\ref{equation: 17 leq F_v(2, 2, 6; 7) leq F_v(3, 6; 7) leq 18}), it remains to be proved that $F_v(3, 6; 7) \neq 17$. Since $\mH(3, 6; 7) \subseteq \mH(2, 2, 6; 7)$ (see (\ref{equation: G arrowsv (a_1, ..., a_s) Rightarrow G arrowsv (a_1, ..., a_(i - 1), t, a_i - t, a_(i + 1), ..., a_s)})), but $G_1 \not\in \mH(3, 6; 7)$, we come to the conclusion that $H_v(3, 6; 7; 17) = \emptyset$ and $F_v(3, 6; 7) = 18$.
\qed

\section{Proof of Theorem \ref{theorem: F_v(a_1, ..., a_s; m - 1) = ..., max set(a_1, ..., a_s) = 6} (a)}

We will do the proof with the help of the following
\begin{theorem}
\label{theorem: rp}
\cite{BN15a}
Let $\rp(p) = \rp$ be the smallest positive integer for which
\begin{equation*}
\min_{r \geq 2} \set{F_v(2_r, p; r + p - 1) - r} = F_v(2_{\rp}, p; \rp + p - 1) - \rp.
\end{equation*}

Then:
\begin{flalign*}
F_v(2_r, p; r + p - 1) = F(2_{\rp}, p; \rp + p - 1) + r - \rp, \quad r \geq \rp. && \tag{a}
\end{flalign*}
\begin{flalign*}
\mbox{if $\rp = 2$, then} && \tag{b}
\end{flalign*}

$F_v(2_r, p; r + p - 1) = F_v(2, 2, p; p + 1) + r - 2, \quad r \geq 2$

\begin{flalign*}
\mbox{if $\rp > 2$ and $G$ is an extremal graph in $\mH (2_{\rp}, p; \rp + p - 1)$, then} && \tag{c}
\end{flalign*}

$G \arrowsv (2, \rp + p - 2).$

\begin{flalign*}
\rp < F_v(2, 2, p; p + 1) - 2p. && \tag{d}
\end{flalign*}
\end{theorem}

Theorem \ref{theorem: rp} is proved in \cite{BN15a} as Theorem 5.2. We will note that the proof of Theorem \ref{theorem: rp} is analogous to that of Theorem \ref{theorem: rpp} from this paper.

In relation to Theorem \ref{theorem: rp}(b) in \cite{BN15a} we formulate
\begin{conjecture}
	\label{conjecture: rp(p) = 2, p geq 4}
	If $p \geq 4$, then
	\begin{equation*}
		\min_{r \geq 2} \set{F_v(2_r, p; r + p - 1) - r} = F_v(2, 2, p; p + 1) - 2,
	\end{equation*}
	and therefore
	\begin{equation*}
		F_v(2_r, p; r + p - 1) = F_v(2, 2, p; p + 1) + r - 2, \quad r \geq 2.
	\end{equation*}
\end{conjecture}
It is not difficult to see that Conjecture \ref{conjecture: rp(p) = 2, p geq 4} is true if and only if the sequence $\set{F_v(2_r, p; r + p - 1)}$ for fixed $p$ is strictly increasing with respect to $r \geq 2$. Since $F_v(2, 2, 3; 4) = 14$ \cite{Nen00} and \cite{CR06}, and $F_v(a_1, ..., a_s; m - 1) = m + 6$, if $p = 3$ and $m \geq 6$ \cite{Nen02}, we have $\rp(3) = 3$. Since $F_v(2, 2, 4; 5) = 13$ \cite{Nen02}, from Theorem \ref{theorem: F_v(2, 2, p; p + 1) leq 2p + 5, then ...} it follows that $\rp(4) = 2$. The equality $\rp(5) = 2$ is also true, but it does not follow from Theorem \ref{theorem: F_v(2, 2, p; p + 1) leq 2p + 5, then ...}. It is proved with the help of a computer in \cite{BN15a} as Theorem 6.1. Therefore, Conjecture \ref{conjecture: rp(p) = 2, p geq 4} is true when $p = 4$ and $p = 5$ . We will prove that when $p = 6$ Conjecture \ref{conjecture: rp(p) = 2, p geq 4} is also true. More specifically, we will prove
\begin{theorem}
\label{theorem: rp(6) = 2}
$\rp(6) = 2$ and therefore $F_v(2_r, 6; r + 5) = r + 15, \ r \geq 2$.
\end{theorem}
Before proving Theorem \ref{theorem: rp(6) = 2} we will prove
\begin{theorem}
\label{theorem: F_v(2, 2, p; p + 1) leq 2p + 5, then ...}
Let $F_v(2, 2, p; p + 1) \leq 2p + 5$. Then $\rp(p) = 2$ and
\begin{equation*}
F_v(2_r, p; r + p - 1) = F_v(2, 2, p; p + 1) + r - 2, \ r \geq 2.
\end{equation*}
\end{theorem}
\begin{proof}
From Theorem \ref{theorem: rp}(b) it follows that it is enough to prove the equality $\rp(p) = 2$. According to (\ref{equation: m + p + 2 leq F_v(a_1, ..., a_s; m - 1) leq m + 3p}), $F_v(2, 2, p; p + 1) \geq 2p + 4$. Therefore, only the following two cases are possible:\\

\emph{Case 1.} $F_v(2, 2, p; p + 1) = 2p + 4$. According to (\ref{equation: m + p + 2 leq F_v(a_1, ..., a_s; m - 1) leq m + 3p})

$F_v(2_r, p; r + p - 1) \geq m + p + 2 = r + 2p + 2$.\\
Therefore,

$F_v(2_r, p; r + p - 1) - r \geq 2p + 2 = F_v(2, 2, p; p + 1) - 2, \ r \geq 2$,\\
and we have $\rp(p) = 2$.

\emph{Case 2.} $F_v(2, 2, p; p + 1) = 2p + 5$. From Theorem \ref{theorem: F_v(a_1, ..., a_s; m - 1) geq m + p + 3} we have $F_v(2_r, p; r + p - 1) \geq r + 2p + 3, \ r \geq 2$. From this inequality we obtain

$F_v(2_r, p; r + p - 1) - r \geq 2p + 3 = F_v(2, 2, p; p + 1) - 2, \ r \geq 2$.\\
Therefore, in this case we also have $\rp(p) = 2$.
\end{proof}

\begin{remark}
\label{remark: F_v(2, 2, p; p + 1) = 2p + 4?}
It is unknown whether the first case is possible, i.e. if $F_v(2, 2, p; p + 1) = 2p + 4$ for some $p$. If $p \leq 7$ this equality is not true.
\end{remark}

\subsection*{Proof of Theorem \ref{theorem: rp(6) = 2}}

According to Theorem \ref{theorem: F_v(2, 2, 6; 7) = 17 and F_v(3, 6; 7) = 18}, $F_v(2, 2, 6; 7) = 17$. From this fact and Theorem \ref{theorem: F_v(2, 2, p; p + 1) leq 2p + 5, then ...} we obtain $\rp(6) = 2$ and the equality $F_v(2_r, 6; r + 5) = r + 15, \ r \geq 2$.
\qed

\subsection*{Proof of Theorem \ref{theorem: F_v(a_1, ..., a_s; m - 1) = ..., max set(a_1, ..., a_s) = 6} (a)}

Since $a_1 = ... = a_{s-1} = 2$ and $a_s = 6$ we have $m = s + 5$ and therefore

$F_v(a_1, ..., a_s; m - 1) = F_v(2_{s - 1}, 6; m - 1) = F_v(2_{m - 6}, 6; m - 1)$.\\
From Theorem \ref{theorem: rp(6) = 2} it now follows that $F_v(a_1, ..., a_s; m - 1) = m + 9$.
\qed

\section{Proof of Theorem \ref{theorem: F_v(a_1, ..., a_s; m - 1) = ..., max set(a_1, ..., a_s) = 6} (b)}

We will need the following

\begin{theorem}
\label{theorem: rpp}
Let $\rpp(p) = \rpp$ be the smallest positive integer for which
\begin{equation*}
\min\set{F_v(2_r, 3, p; r + p + 1) - r} = F_v(2_{\rpp}, 3, p; \rpp + p + 1) - \rpp
\end{equation*}
Then
\begin{flalign*}
F_v(2_r, 3, p; r + p + 1) = F(2_{\rpp}, 3, p; \rpp + p + 1) + r - \rpp, \quad r \geq \rpp. && \tag{a}
\end{flalign*}
\begin{flalign*}
\mbox{if $\rpp = 0$, then} && \tag{b}
\end{flalign*}

$F_v(2_r, 3, p; r + p + 1) = F_v(3, p; p + 1) + r, \quad r \geq 0$

\begin{flalign*}
\mbox{if $\rpp > 0$ and $G$ is an extremal graph in $\mH (2_{\rpp}, 3, p; \rpp + p + 1)$, then} && \tag{c}
\end{flalign*}

$G \arrowsv (2, \rpp + p).$

\begin{flalign*}
\rpp < F_v(3, p; p + 1) - 2p - 2. && \tag{d}
\end{flalign*}
\end{theorem}

\begin{proof}
(a) According to the definition of $\rpp = \rpp(p)$ we have

$F_v(2_r, 3, p; r + p + 1) \geq F_v(2_{\rpp}, 3, p; \rpp + p + 1) + r - \rpp, \ r \geq 0$.\\
Now we will prove that if $r \geq \rpp$ the opposite inequality is also true. Let us note that if $G \arrowsv (a_1, ..., a_s)$, then $K_1 + G \arrowsv (2, a_1, ..., a_s)$. It follows
\begin{equation}
\label{equation: G arrowsv (a_1, ..., a_s) Rightarrow K_t + G arrowsv (2_t, a_1, ..., a_s)}
G \arrowsv (a_1, ..., a_s) \Rightarrow K_t + G \arrowsv (2_t, a_1, ..., a_s).
\end{equation}
Let $G \in \mH_{extr}(2_{\rpp}, 3, p; \rpp + p + 1)$. Then from (\ref{equation: G arrowsv (a_1, ..., a_s) Rightarrow K_t + G arrowsv (2_t, a_1, ..., a_s)}) it follows that $K_{r - \rpp} + G \in \mH(2_r, 3, p; r + p + 1), \ r \geq \rpp$. Therefore

$F_v(2_r, 3, p; r + p + 1) \leq \abs{\V(K_{r - \rpp} + G)} = F_v(2_{\rpp}, 3, p; \rpp + p + 1) + r - \rpp, \ r \geq \rpp$.\\
Thus, (a) is proved.\\

(b) If $\rpp(p) = 0$, then obviously the equality (b) follows from (a).\\

(c) Assume the opposite is true and let $G$ be an extremal graph in $\mH(2_{\rpp}, 3, p; \rpp + p + 1)$, such that $\V(G) = V_1 \cup V_2$ where $V_1$ is an independent set and $V_2$ does not contain $(\rpp + p)$-clique. We can assume that $V_1 \neq \emptyset$. Let $G_1 = \G[V_2] = G - V_1$. Then $\omega(G_1) < \rpp + p$ and since $\rpp \geq 1$, from Proposition \ref{proposition: G - A arrowsv (a_1, ..., a_(i - 1), a_i - 1, a_(i + 1_, ..., a_s)} it follows that $G_1 \arrowsv (2_{\rpp - 1}, 3, p)$. Therefore $G_1 \in \mH(2_{\rpp - 1}, 3, p; \rpp + p)$ and

$\abs{\V(G)} - 1 \geq \abs{\V(G_1)} \geq F_v(2_{\rpp - 1}, 3, p; \rpp + p)$.\\
Since $\abs{\V(G)} = F_v(2_{\rpp}, 3, p; \rpp + p + 1)$ we obtain

$F_v(2_{\rpp - 1}, 3, p; \rpp + p) - (\rpp - 1) \leq F_v(2_{\rpp}, 3, p; \rpp + p + 1) - \rpp$,\\
which contradicts the definition of $\rpp$.\\

(d) According to (\ref{equation: m + p + 2 leq F_v(a_1, ..., a_s; m - 1) leq m + 3p}) $F_v(3, p; p + 1) \geq 2p + 4$ and therefore in the case $\rpp = 0$ the inequality holds. Let $\rpp > 0$ and $G$ be an extremal graph in $\mH(2_{\rpp}, 3, p; \rpp + p + 1)$. According to (\ref{equation: G arrowsv (a_1, ..., a_s) Rightarrow chi(G) geq m})
\begin{equation}
\label{equation: chi(G) geq rpp + p + 2}
\chi(G) \geq \rpp + p + 2.
\end{equation}
According to (c) and Theorem \ref{theorem: F_v(a_1, ..., a_s; m) = m + p}

$\abs{\V(G)} \geq 2\rpp + 2p + 1$.\\
Since $\chi(\overline{C}_{2\rpp + 2p + 1}) = \rpp + p + 1$, from (\ref{equation: chi(G) geq rpp + p + 2}) it follows $G \neq \overline{C}_{2\rpp + 2p + 1}$. By Theorem \ref{theorem: F_v(a_1, ..., a_s; m) = m + p}(b)

$\abs{\V(G)} = F_v(2_{\rpp}, 3, p; \rpp + p + 1) \geq 2\rpp + 2p + 2$.\\
Since $\rpp > 0$, we have

$F_v(2_{\rpp}, 3, p; \rpp + p + 1) - \rpp < F_v(3, p; p + 1).$\\
From these inequalities we can easily see that

$\rpp < F_v(3, p; p + 1) - 2p - 2$.
\end{proof}

Since $F_v(3, 3; 4) = 14$, from (\ref{equation: F_v(a_1, ..., a_s, m - 1) = ...}) we obtain $\rpp(3) = 1$. Also from (\ref{equation: F_v(a_1, ..., a_s, m - 1) = ...}) we see that $\rpp(4) = 0$ and $\rpp(5) = 0$. We suppose the following conjecture is true
\begin{conjecture}
\label{conjecture: rpp(p) = 0, p geq 4}
If $p \geq 4$, then
	\begin{equation*}
		\min \set{F_v(2_r, 3, p; r + p - 1) - r} = F_v(3, p; p + 1),
	\end{equation*}
and therefore
\begin{equation*}
F_v(2_r, 3, p; r + p + 1) = F_v(3, p; p + 1) + r.
\end{equation*}
\end{conjecture}
It is not difficult to see that Conjecture \ref{conjecture: rpp(p) = 0, p geq 4} is true if and only if the sequence $\set{F_v(2_r, 3, p; r + p + 1)}$ for fixed $p$ is strictly increasing with respect to $r$. We will prove that when $p = 6$ Conjecture \ref{conjecture: rpp(p) = 0, p geq 4} is also true. The Theorem \ref{theorem: F_v(a_1, ..., a_s; m - 1) = ..., max set(a_1, ..., a_s) = 6} (b) follows easily from this fact.

\begin{theorem}
\label{theorem: rpp(6) = 0}
$\rpp(6) = 0$.
\end{theorem}

\begin{proof}
From Theorem \ref{theorem: rpp} (d) we obtain $\rpp(6) < 4$. Therefore we have to prove $\rpp(6) \neq 1$, $\rpp(6) \neq 2$ and $\rpp(6) \neq 3$. Since $F_v(3, 6; 7) = 18$, we have to prove the inequalities $F_v(2, 3, 6; 8) > 18$, $F_v(2, 2, 3, 6; 9) > 19$ and $F_v(2, 2, 2, 3, 6; 10) > 20$. We will prove these inequalities with the help of a computer. From (\ref{equation: G arrowsv (a_1, ..., a_s) Rightarrow K_t + G arrowsv (2_t, a_1, ..., a_s)}) (t = 1) it is easy to see that $F_v(2_{r - 1}, 3; p) + 1 \geq F_v(2_r, 3; p + 1)$ and therefore it is enough to prove $F_v(2, 2, 2, 3, 6; 10) > 20$. We shall present the proof of this inequality only, but we also checked the other two inequalities in the same way with a computer in order to obtain more information, which is presented in Appendix A.

Similarly to the proof of Theorem \ref{theorem: abs(mathcal(H)(2, 2, 6; 7; 17)) = 3}, we shall use Algorithm \ref{algorithm: mH_(max)(a_1, ..., a_s; q; n), alpha(G) = 2} and Algorithm \ref{algorithm: mH_(max)(a_1, ..., a_s; q; n)} to prove that $\mH(2, 2, 2, 3, 6; 10; 20) = \emptyset$. According to (\ref{equation: G in mH(a_1, ..., a_s; m - 1; n) Rightarrow alpha(G) leq n - m - p + 1}), there are no graphs in $\mH(2, 2, 2, 3, 6; 10; 20)$ with independence number greater than 4.

By Theorem \ref{theorem: F_v(a_1, ..., a_s; m) = m + p}, $K_3 + \overline{C}_{13}$ is the only graph in $\mH(2, 2, 3, 6; 10; 16)$. Starting from $\mH_{max}(2, 2, 3, 6; 10; 16) = \set{K_3 + \overline{C}_{13}}$, by applying Algorithm \ref{algorithm: mH_(max)(a_1, ..., a_s; q; n)} ($r = 4; t = 4$) we show that there are no graphs in $\mH_{max}(2, 2, 2, 3, 6; 10; 20)$ with independence number 4.

The next step is to prove that there are no graphs in $\mH_{max}(2, 2, 2, 3, 6; 10; 20)$ with independence number 3. The only graph in $\mH_{max}(6; 10; 9)$ is $K_9$. Starting from $\mH_{max}(6; 10; 9) = \set{K_9}$ by applying Algorithm \ref{algorithm: mH_(max)(a_1, ..., a_s; q; n)} ($r = 2; t = 3$) we successively obtain all graphs with independence number not greater than 3 in $\mH_{max}(2, 6; 10; 11)$, $\mH_{max}(3, 6; 10; 13)$, $\mH_{max}(2, 3, 6; 10; 15)$, $\mH_{max}(2, 2, 3, 6; 10; 17)$. In the end, we apply Algorithm \ref{algorithm: mH_(max)(a_1, ..., a_s; q; n)} ($r = 3; t = 3$) to the obtained graphs in $\mH_{max}(2, 2, 3, 6; 10; 17)$ with independence number not greater than 3 to show that there are no graphs in $\mH_{max}(2, 2, 2, 3, 6; 10; 20)$ with independence number 3.

Finally, we prove that there are no graphs in $\mH_{max}(2, 2, 2, 3, 6; 10; 20)$ with independence number 2. The only graph in $\mH_{max}(6; 10; 10)$ is $K_{10} - e$. Starting from $\mH_{max}(6; 10; 10) = \set{K_{10} - e}$ by applying Algorithm \ref{algorithm: mH_(max)(a_1, ..., a_s; q; n), alpha(G) = 2} we successively obtain all graphs with independence number 2 in $\mH_{max}(2, 6; 10; 12)$, $\mH_{max}(3, 6; 10; 14)$, $\mH_{max}(2, 3, 6; 10; 16)$, $\mH_{max}(2, 2, 3, 6; 10; 18)$ and $\mH_{max}(2, 2, 2, 3, 6; 10; 20)$. As a result, no graphs in $\mH_{max}(2, 2, 2, 3, 6; 10; 20)$ with independence number 2 were obtained.

Thus, we proved $\mH_{max}(2, 2, 2, 3, 6; 10; 20) = \emptyset$ and therefore $F_v(2, 2, 2, 3, 6; 10) > 20$ and $\rpp(6) = 0$.

The numbers of graphs obtained in each step are shown in Table \ref{table: finding all graphs in mathcal(H)(2, 2, 2, 3, 6; 10; 20)} (see also Table \ref{table: finding all graphs in mathcal(H)(2, 3, 6; 8; 18)} and Table \ref{table: finding all graphs in mathcal(H)(2, 2, 3, 6; 9; 19)}).
\end{proof}

\subsection*{Proof of Theorem \ref{theorem: F_v(a_1, ..., a_s; m - 1) = ..., max set(a_1, ..., a_s) = 6} (b)}

According to Theorem \ref{theorem: rpp(6) = 0} and Theorem \ref{theorem: rpp} (b) we have $F_v(2_{m - 8}, 3, 6; m - 1) = m + 10$. From (\ref{equation: F_v(2_(m - p - 2), 3, p; m - 1) leq F_v(a_1, ..., a_s; m - 1)}) it now follows $F_v(a_1, ..., a_s; m - 1) \geq m + 10$. The opposite inequality is true according to (\ref{equation: m + 9 leq F_v(a_1, ..., a_s) leq m + 10}) (see also the Main Theorem from \cite{BN15b}).
\qed

\section{Finding all graphs in $\mH(2, 2, 6; 7; 18)$ and proofs of Theorem \ref{theorem: F_v(a_1, ..., a_s; 7) geq F_v(2_(m - 6), 6; 7) geq 3m - 5} and Theorem \ref{theorem: F_v(a_1, ..., a_s; 7) leq F_v(6, 6; 7) leq 60}}

\begin{theorem}
	\label{theorem: abs(mathcal(H)(2, 2, 6; 7; 18)) = 76515}
	$\abs{\mH(2, 2, 6; 7; 18)} = 76515$.
\end{theorem}

\begin{proof}
Similarly to the proof of Theorem \ref{theorem: abs(mathcal(H)(2, 2, 6; 7; 17)) = 3}, we will find all graphs in $\mH(2, 2, 6; 7; 18)$ with the help of a computer. Some of the graphs that we obtain in the steps of this proof were already obtained in the proof of Theorem \ref{theorem: abs(mathcal(H)(2, 2, 6; 7; 17)) = 3} (compare Table \ref{table: finding all graphs in H(2, 2, 6; 7; 17)} to Table \ref{table: finding all graphs in H(2, 2, 6; 7; 18)}).

Let $G \in \mH(2, 2, 6; 7; 18)$. Clearly $\alpha(G) \geq 2$, and according to (\ref{equation: G in mH(a_1, ..., a_s; m - 1; n) Rightarrow alpha(G) leq n - m - p + 1}), $\alpha(G) \leq 5$.
	
First, we prove that there are no graphs in $\mH_{max}(2, 2, 6; 7; 18)$ with independence number 5. According to Theorem \ref{theorem: F_v(a_1, ..., a_s; m) = m + p}, $\overline{C}_{13}$ is the only graph in $\mH(2, 6; 7; 13)$. Starting from $\mH_{max}(2, 6; 7; 13) = \set{\overline{C}_{13}}$, by applying Algorithm \ref{algorithm: mH_(max)(a_1, ..., a_s; q; n)} ($r = 5; t = 5$) we show that there are no graphs in $\mH(2, 2, 6; 7; 18)$ with independence number 5.
	
Now, we shall prove that there are no graphs in $\mH_{max}(2, 2, 6; 7; 17)$ with independence number 4. Starting from $\mH_{max}(3; 7; 6) = \set{K_6}$, by applying Algorithm \ref{algorithm: mH_(max)(a_1, ..., a_s; q; n)} ($r = 2; t = 4$) we successively obtain all graphs with independence number not greater than 4 in $\mH_{max}(4; 7; 8)$, $\mH_{max}(5; 7; 10)$, $\mH_{max}(6; 7; 12)$, $\mH_{max}(2, 6; 7; 14)$. By applying Algorithm \ref{algorithm: mH_(max)(a_1, ..., a_s; q; n)} ($r = 4; t = 4$) to the obtained graphs in $\mH_{max}(2, 6; 7; 14)$ with independence number not greater than 4 we conclude that there are no graphs in $\mH_{max}(2, 2, 6; 7; 18)$ with independence number 4.
	
Next, we find all graphs in $\mH_{max}(2, 2, 6; 7; 18)$ with independence number 3. Starting from $\mH_{max}(3; 7; 7) = \set{K_7 - e}$, by applying Algorithm \ref{algorithm: mH_(max)(a_1, ..., a_s; q; n)} ($r = 2; t = 3$) we successively obtain all graphs with independence number not greater than 3 in $\mH_{max}(4; 7; 9)$, $\mH_{max}(5; 7; 11)$, $\mH_{max}(6; 7; 13)$, $\mH_{max}(2, 6; 7; 15)$. By applying Algorithm \ref{algorithm: mH_(max)(a_1, ..., a_s; q; n)} ($r = 3; t = 3$) to the obtained graphs in $\mH_{max}(2, 6; 7; 15)$ with independence number not greater than 3 we obtain all 308 graphs in $\mH_{max}(2, 2, 6; 7; 18)$ with independence number 3.

The last, and computationally most difficult step, is to find all graphs in $\mH_{max}(2, 2, 6; 7; 18)$ with independence number 2. It is easy to see that $\mH_{max}(3; 7; 8) = \set{\overline{K}_3 + K_5, C_4 + K_4}$ and therefore $C_4 + K_4$ is the only graph in $\mH_{max}(3; 7; 8)$ with independence number 2. Starting from $\set{C_4 + K_4}$, by applying Algorithm \ref{algorithm: mH_(max)(a_1, ..., a_s; q; n), alpha(G) = 2} we successively obtain all graphs with independence number 2 in $\mH_{max}(4; 7; 10)$, $\mH_{max}(5; 7; 12)$, $\mH_{max}(6; 7; 14)$, $\mH_{max}(2, 6; 7; 16)$ and $\mH_{max}(2, 2, 6; 7; 18)$. As a result, we find all 84 graphs in $\mH_{max}(2, 2, 6; 7; 18)$ with independence number 2.

Thus, we obtained all 392 graphs in $\mH_{max}(2, 2, 6; 7; 18)$. By removing edges from these graphs we find all 76 515 graphs in $\mH(2, 2, 6; 7; 18)$. Some properties of these graphs are listed in Table \ref{table: H(2, 2, 6; 7; 18) properties}. The number of maximal $K_7$-free graphs and $(+K_6)$-graphs obtained in each step of the proof is shown in Table \ref{table: finding all graphs in H(2, 2, 6; 7; 18)}.
\end{proof}

\begin{table}[h]
	\centering
	\resizebox{\textwidth}{!}{
		\begin{tabular}{ | p{2.0cm} | p{2.0cm} | p{2.0cm} | p{2.0cm} | p{2.0cm} | p{2.0cm} | p{2.0cm} | }
			\hline
			$|\E(G)|$	\hfill $\#$		& $\delta(G)$ \hfill $\#$	& $\Delta(G)$ \hfill $\#$	& $\alpha(G)$ \hfill $\#$	& $\chi(G)$ \hfill $\#$		& $|Aut(G)|$ \hfill $\#$\\
			\hline
			106	\hfill 1				& 0		\hfill 3			& 13	\hfill 65			& 2	\hfill 290				& 8	\hfill 84				& 1		\hfill 72 335	\\
			107	\hfill 4				& 1		\hfill 20			& 14	\hfill 76 450		& 3	\hfill 76 225			& 9	\hfill 76 431			& 2		\hfill 3 699	\\
			108	\hfill 19				& 2		\hfill 124			& 		\hfill 				& 	\hfill 					& 	\hfill 					& 4		\hfill 430		\\
			109	\hfill 88				& 3		\hfill 571			& 		\hfill 				& 	\hfill 					& 	\hfill 					& 8		\hfill 33		\\
			110	\hfill 369				& 4		\hfill 1 943		& 		\hfill 				& 	\hfill 					& 	\hfill 					& 10	\hfill 2		\\
			111	\hfill 1 240			& 5		\hfill 4 986		& 		\hfill 				& 	\hfill 					& 	\hfill 					& 16	\hfill 2		\\
			112	\hfill 3 303			& 6		\hfill 9 826		& 		\hfill 				& 	\hfill 					& 	\hfill 					& 20	\hfill 6		\\
			113	\hfill 6 999			& 7		\hfill 14 896		& 		\hfill 				& 	\hfill 					& 	\hfill 					& 24	\hfill 1		\\
			114	\hfill 11 780			& 8		\hfill 17 057		& 		\hfill 				& 	\hfill 					& 	\hfill 					& 36	\hfill 1		\\
			115	\hfill 15 603			& 9		\hfill 14 288		& 		\hfill 				& 	\hfill 					& 	\hfill 					& 40	\hfill 6		\\
			116	\hfill 15 956			& 10	\hfill 8 397		& 		\hfill 				& 	\hfill 					& 	\hfill 					& 		\hfill			\\
			117	\hfill 12 266			& 11	\hfill 3 504		& 		\hfill 				& 	\hfill 					& 	\hfill 					& 		\hfill			\\
			118	\hfill 6 575			& 12	\hfill 876			& 		\hfill 				& 	\hfill 					& 	\hfill 					& 		\hfill			\\
			119	\hfill 2 044			& 13	\hfill 24			& 		\hfill 				& 	\hfill 					& 	\hfill 					& 		\hfill			\\
			120	\hfill 261				& 		\hfill 				& 		\hfill 				& 	\hfill 					& 	\hfill 					& 		\hfill			\\
			121	\hfill 7				& 		\hfill 				& 		\hfill 				& 	\hfill 					& 	\hfill 					& 		\hfill			\\
			\hline
		\end{tabular}
	}
	\caption{Some properties of the graphs in $\mH(2, 2, 6; 7; 18)$}
	\label{table: H(2, 2, 6; 7; 18) properties}
\end{table}

We check with a computer that among the 76 515 graphs in $\mH(2, 2, 6; 7; 18)$, only the graph $G_4$ (see Figure \ref{figure: H(3, 6; 7; 17)}) belongs to $\mH(3, 6; 7; 18)$. This is the graph that gives the upper bound $F_v(3, 6; 7) \leq 18$ in \cite{SXP09}. Thus, we proved the following

\begin{theorem}
\label{theorem: abs(mathcal(H)(3, 6; 7; 18)) = 1}
$\abs{\mH(3, 6; 7; 18)} = 1$ and $\mH_{extr}(3, 6; 7) = \mH(3, 6; 7; 18) = \set{G_4}$.
\end{theorem}

Let us note, that $\chi(G_4) = 9$ and for this graph the inequality (\ref{equation: G arrowsv (a_1, ..., a_s) Rightarrow chi(G) geq m}) is strict. However, from Theorem \ref{theorem: abs(mathcal(H)(3, 6; 7; 18)) = 1} it follows that in this special case Conjecture \ref{conjecture: chi(G) leq m + 1} is true.

\begin{figure}[!hb]
	\centering
	\begin{subfigure}{0.5\textwidth}
		\centering
		\includegraphics[height=236px,width=118px]{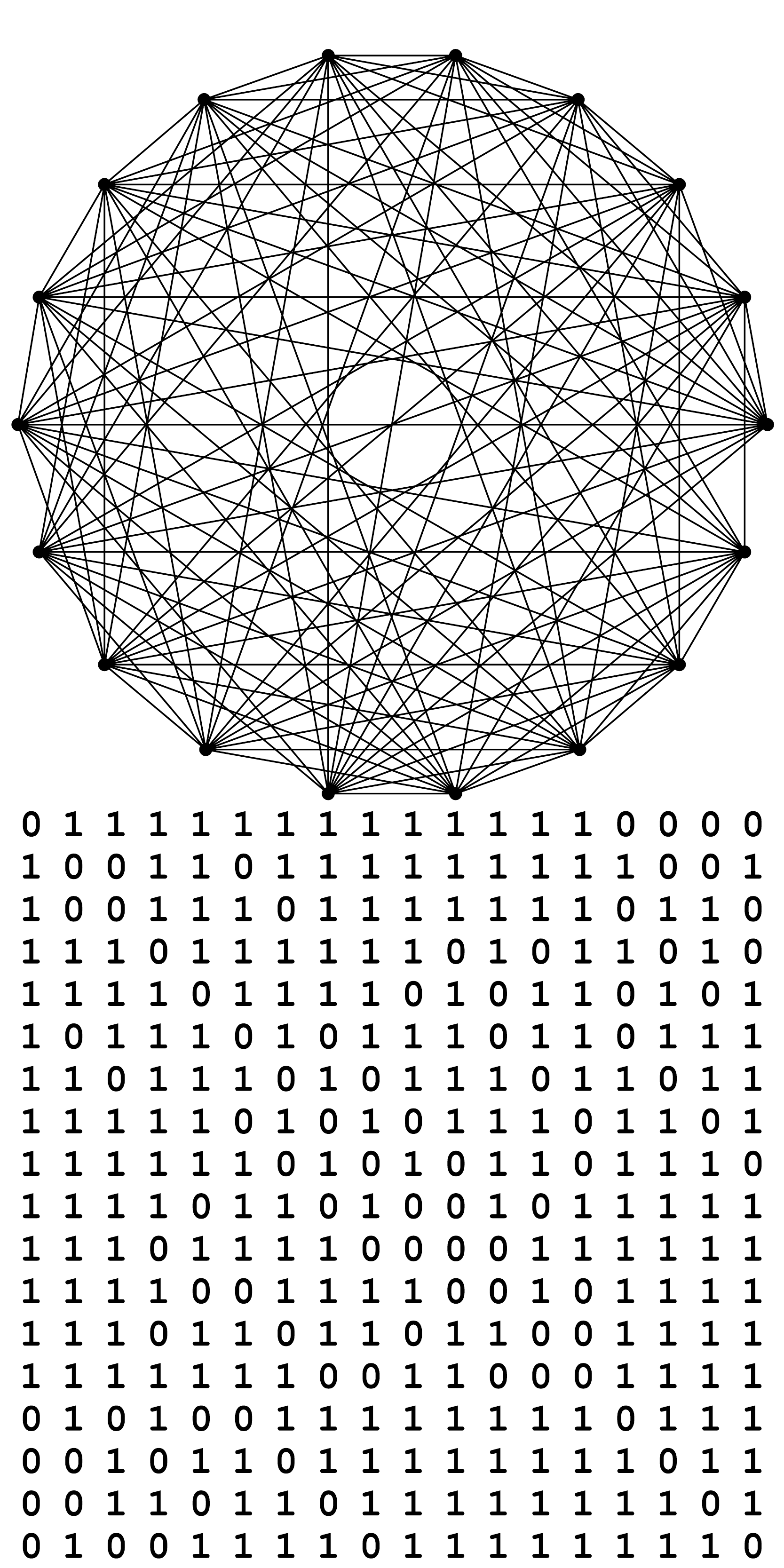}
		\caption*{\emph{$G_{4}$}}
		\label{figure: G_4}
	\end{subfigure}%
	\caption{The only graph $G_4 \in \mH(3, 6; 7; 18)$ from \cite{SXP09}}
	\label{figure: H(3, 6; 7; 17)}
\end{figure}

There are two 13-regular graphs in $\mH(2, 2, 6; 7; 18)$, one of them being $G_4$. The graph $G_4$ is the only vertex transitive graph in $\mH(2, 2, 6; 7; 18)$ and it has 36 automorphisms. The other 13-regular graph has 24 automorphisms.

Let us also note that there are 2 467 vertex critical graphs in $\mH(2, 2, 6; 7; 18)$. We obtained all 74048  non-critical graphs in another way by adding one vertex to the graphs in $\mH(2, 2, 6; 7; 17)$. This also testifies to the correctness of our implementation.

\subsection*{Proof of Theorem \ref{theorem: F_v(a_1, ..., a_s; 7) geq F_v(2_(m - 6), 6; 7) geq 3m - 5}}

According to Theorem \ref{theorem: F_v(2_(m - p), p; q) leq F_v(a_1, ..., a_s; q) leq wFv(m)(p)(q)} from this paper, $F_v(a_1, ..., a_s; 7) \geq F_v(2_{m - 6}, 6; 7)$. We shall prove by induction that $F_v(2_{m - 6}, 6; 7) \geq 3m - 5, m \geq 9$. 

The base case is $m = 9$, i.e. we have to prove that $F_v(2, 2, 2, 6; 7) \geq 22$. We will show that $\mH(2, 2, 2, 6; 7; 21) = \emptyset$ From $F_v(2, 2, 6; 7) = 17$ and Proposition \ref{proposition: G - A arrowsv (a_1, ..., a_(i - 1), a_i - 1, a_(i + 1_, ..., a_s)} it follows that there are no graphs in $\mH(2, 2, 2, 6; 7; 21)$ with independence number greater than 4. All graphs in $\mH(2, 2, 6; 7; 17)$ have independence number 2 (see Table \ref{table: H(2, 2, 6; 7; 17) properties}) and all graphs in $\mH(2, 2, 6; 7; 18)$ have independence number 2 or 3 (see Table \ref{table: H(2, 2, 6; 7; 18) properties}). No graphs are obtained by applying Algorithm \ref{algorithm: mH_(max)(a_1, ..., a_s; q; n)} ($r = 4, t = 4$) to the graphs in $\mH_{max}(2, 2, 6; 7; 17)$, or by applying Algorithm \ref{algorithm: mH_(max)(a_1, ..., a_s; q; n)} ($r = 3, t = 3$) to the graphs in $\mH_{max}(2, 2, 6; 7; 18)$. From Theorem \ref{theorem: algorithm mH_(max)(a_1, ..., a_s; q; n)} it follows that there are no graphs in $\mH(2, 2, 2, 6; 7; 21)$ with independence number 3 or 4. It remains to be proved that there are no graphs in $\mH(2, 2, 2, 6; 7; 21)$ with independence number 2. All 21-vertex graphs $G$ for which $\alpha(G) < 3$ and $\omega(G) < 7$ are known and are available on \cite{McK_r}. There are 1 118 436 such graphs $G$, and with the help of the computer we check that none of these graphs belong to $\mH(2, 2, 2, 6; 7)$. Thus, we proved $\mH(2, 2, 2, 6; 7; 21) = \emptyset$ and $F_v(2, 2, 2, 6; 7) \geq 22$.

Now suppose that for all $m'$ such that $9 \leq m' < m$ we have $F_v(2_{m' - 6}, 6; 7) \geq 3m' - 5$. Let $G \in \mH(2_{m - 6}, 6; 7)$ and $\abs{\V(G)} = F_v(2_{m - 6}, 6; 7)$. From the base case it follows that $F_v(2_{m - 6}, 6; 7) > 22$, and since the Ramsey number $R(3, 7) = 23$ we have $\alpha(G) \geq 3$. Let $A$ be an independent set of vertices of $G$ and $\abs{A} = 3$. According to Proposition \ref{proposition: G - A arrowsv (a_1, ..., a_(i - 1), a_i - 1, a_(i + 1_, ..., a_s)}, $G - A \in \mH(2_{m - 7}, 6; 7)$ and therefore
$$F_v(2_{m - 6}, 6; 7) = \abs{\V(G)} \geq F_v(2_{m - 7}, 6; 7) + \abs{A} \geq 3(m - 1) - 5 + 3 = 3m - 5.$$
\qed

\subsection*{Proof of Theorem \ref{theorem: F_v(a_1, ..., a_s; 7) leq F_v(6, 6; 7) leq 60}}

We will prove only (b), since (a) can be proved in the same way. The lower bound is true according to Theorem \ref{theorem: F_v(a_1, ..., a_s; 7) geq F_v(2_(m - 6), 6; 7) geq 3m - 5}. Clearly, we can assume that $a_s = \max\set{a_1, ..., a_s} = 6$. Therefore, from (\ref{equation: G arrowsv (a_1, ..., a_s) Rightarrow G arrowsv (a_1, ..., a_(i - 1), t, a_i - t, a_(i + 1), ..., a_s)}) we obtain the inclusion $\mH(6, 6; 7) \subseteq \mH(a_1, ..., a_s; 7)$ and it follows that $F_v(a_1, ..., a_s; 7) \leq F_v(6, 6; 7)$. Kolev proves in \cite{Kol08} that
$$F_v(a_1, ..., a_s; q + 1) . F_v(b_1, ..., b_s; t + 1) \geq F_v(a_1.b_1, ..., a_s.b_s; qt + 1).$$
Since $F_v(2, 2; 3) = 5$ and $F_v(3, 3; 4) = 14$, \cite{Nen81} and \cite{PRU99}, it follows that
$F_v(6, 6; 7) \leq F_v(2, 2; 3) . F_v(3, 3; 4) = 70.$ In (a) instead of $F_v(3, 3; 4) = 14$ we use $F_v(2, 3; 4) = 7$ (see Theorem \ref{theorem: F_v(a_1, ..., a_s; m) = m + p}).
\qed

\section{Proof of Theorem \ref{theorem: F_v(2, 2, 7; 8) = 20}}

\subsection*{Proof of the lower bound $F_v(2, 2, 7; 8) \geq 20$}

We can prove that $\mH(2, 2, 7; 8; 19) = \emptyset$ using the method from the proof of Theorem \ref{theorem: abs(mathcal(H)(2, 2, 6; 7; 17)) = 3}. Suppose that $G \in \mH(2, 2, 7; 8; 19)$. Clearly $\alpha(G) \geq 2$, and according to (\ref{equation: G in mH(a_1, ..., a_s; m - 1; n) Rightarrow alpha(G) leq n - m - p + 1}), $\alpha(G) \leq 4$.  According to Theorem \ref{theorem: F_v(a_1, ..., a_s; m) = m + p},  $\overline{C}_{15}$ is the only graph in $\mH(2, 7; 8 ; 15)$. By applying Algorithm \ref{algorithm: mH_(max)(a_1, ..., a_s; q; n)}($r = 4; t = 4$) to $\mH_{max}(2, 7; 8 ; 15) = \set{\overline{C}_{15}}$ we prove that there are no graphs in $\mH_{max}(2, 2, 7; 8; 19)$ with independence number 4.

With the help of Algorithm \ref{algorithm: mH_(max)(a_1, ..., a_s; q; n)}($r = 2; t = 3$) we successively obtain all graphs with independence number not greater than 3 in $\mH_{max}(4; 8; 8)$, $\mH_{max}(5; 8; 10)$, $\mH_{max}(6; 8; 12)$, $\mH_{max}(7; 8; 14)$, $\mH_{max}(2, 7; 8; 16)$. Then, we apply Algorithm \ref{algorithm: mH_(max)(a_1, ..., a_s; q; n)}($r = 3; t = 3$) to the obtained graphs in $\mH_{max}(2, 7; 8; 16)$ with independence number not greater than 3 to show that there are no graphs in $\mH_{max}(2, 2, 7; 8; 19)$ with independence number 3.

In the last and computationally most difficult part of the proof, with the help of Algorithm \ref{algorithm: mH_(max)(a_1, ..., a_s; q; n), alpha(G) = 2} we successively obtain all graphs with independence number 2 in $\mH_{max}(4; 8; 9)$, $\mH_{max}(5; 8; 11)$, $\mH_{max}(6; 8; 13)$, $\mH_{max}(7; 8; 15)$, $\mH_{max}(2, 7; 8; 17)$ and $\mH_{max}(2, 2, 7; 8; 19)$. As a result, no graphs in $\mH_{max}(2, 2, 7; 8; 19)$ with independence number 2 are obtained.

Thus, we proved $\mH_{max}(2, 2, 7; 8; 19) = \emptyset$. The number of graphs obtained in each step are shown in Table \ref{table: finding all graphs in H(2, 2, 7; 8; 19)}.
\qed 

\subsection*{Proof of the upper bound $F_v(2, 2, 7; 8) \leq 20$}

We need to construct a 20-vertex graph in $\mH(2, 2, 7, 8; 20)$.
All vertex transitive graphs with up to 31 vertices are known and can be found in \cite{Roy_t}. With the help of a computer we check which of these graphs belong to $\mH(2, 2, 7; 8)$. In this way, we find one 24-vertex graph, one 28-vertex graph and 6 30-vertex graphs in $\mH(2, 2, 7; 8)$.

By removing one vertex from the 24-vertex transitive graph in $\mH(2, 2, 7; 8)$ we obtain 3 23-vertex graphs in $\mH(2, 2, 7; 8)$, and by removing two vertices we obtain 8 22-vertex graphs in $\mH(2, 2, 7; 8)$. We add two edges to one of the 8 22-vertex graphs (the only one with 180 edges) we find one graph in $\mH_{max}(2, 2, 7; 8; 22)$. Using the following Procedure \ref{procedure: populate} we find 1696 more graphs in $\mH_{max}(2, 2, 7; 8; 22)$.

By removing one vertex to the obtained graphs in $\mH_{max}(2, 2, 7; 8; 22)$ we find 22 21-vertex graphs in $\mH(2, 2, 7; 8)$. We add edges to these graphs to obtain 22 graphs in $\mH_{max}(2, 2, 7; 8; 21)$. Then, we apply Procedure \ref{procedure: populate} twice to obtain 15259 more graphs in $\mH_{max}(2, 2, 7; 8; 21)$.

By removing one vertex to the obtained graphs in $\mH_{max}(2, 2, 7; 8; 21)$ we find 9 20-vertex graphs in $\mH(2, 2, 7; 8)$. Again, by successively applying Procedure \ref{procedure: populate} we obtain 39 graphs in $\mH_{max}(2, 2, 7; 8; 20)$. One of these graphs is the graph $G_5$ shown on Figure \ref{figure: H(2, 2, 7; 8; 20)}. Later, we shall use the graph $G_5$ in the proof of Theorem \ref{theorem: F_v(a_1, ..., a_s; m - 1) leq m + 12, max set(a_1, ..., a_s) = 7}.
\qed

\clearpage
\begin{procedure}
	\label{procedure: populate}
	\cite{BN15a}
	Extending a set of maximal graphs in $\mH(a_1, ..., a_s; q; n)$.
	
	1. Let $\mathcal{A}$ be a set of maximal graphs in $\mH(a_1, ..., a_s; q; n)$.
	
	2. By removing edges from the graphs in $\mathcal{A}$, find all their subgraphs which are in $\mH(a_1, ..., a_s; q; n)$. This way a set of non-maximal graphs in $\mH(a_1, ..., a_s; q; n)$ is obtained.
	
	3. Add edges to the non-maximal graphs to find all their supergraphs which are maximal in $\mH(a_1, ..., a_s; q; n)$. Extend the set $\mathcal{A}$ by adding the new maximal graphs.
\end{procedure}

\begin{figure}[h]
	\centering
	\begin{subfigure}{\textwidth}
		\centering
		\includegraphics[height=320px,width=160px]{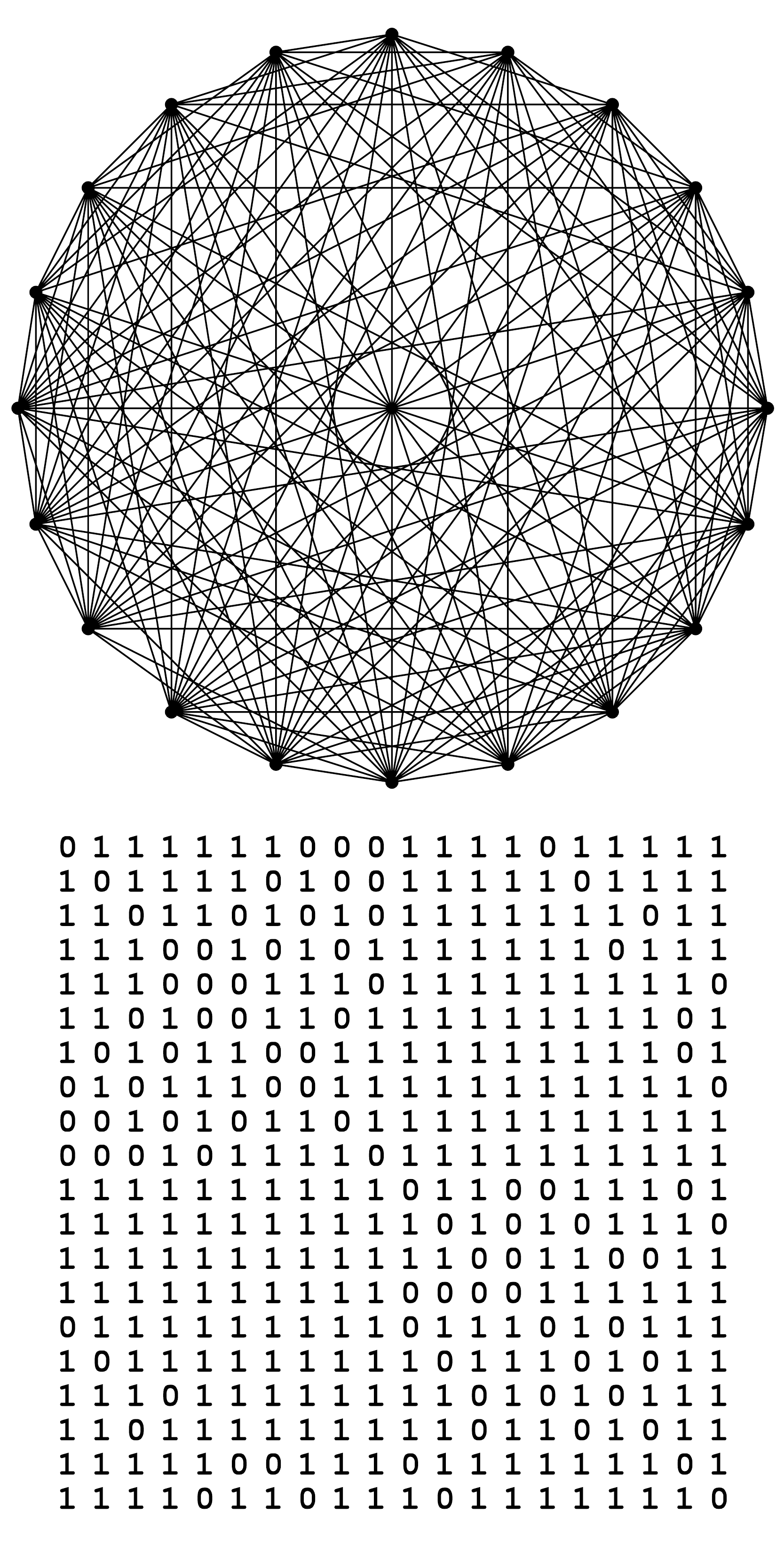}
		\caption*{\emph{$G_{5}$}}
		\label{figure: G_5}
	\end{subfigure}%
	\caption{20-vertex graph $G_5 \in \mH(2, 2, 7; 8)$}
	\label{figure: H(2, 2, 7; 8; 20)}
\end{figure}

\section{Proof of Theorem \ref{theorem: F_v(a_1, ..., a_s; m - 1) leq m + 12, max set(a_1, ..., a_s) = 7}}

In \cite{BN15a}, we define a modification of the vertex Folkman numbers $F_v(a_1, ..., a_s; q)$ with the help of which we obtain the upper bound in (\ref{equation: m + 9 leq F_v(a_1, ..., a_s) leq m + 10}).

\begin{definition}
	\label{definition: G overset(v)(rightarrow) uni(m)(p)}
	\cite{BN15a}
	Let $G$ be a graph and let $m$ and $p$ be positive integers. The expression
	\begin{equation*}
	G \arrowsv \uni{m}{p}
	\end{equation*}
	means that for every choice of positive integers $a_1, ..., a_s$ ($s$ is not fixed), such that $m = \sum\limits_{i=1}^s (a_i - 1) + 1$ and $\max\set{a_1, ..., a_s} \leq p$, we have
	\begin{equation*}
	G \arrowsv (a_1, ..., a_s).
	\end{equation*}
\end{definition}

In \cite{BN15a} we also define the following notations (see also \cite{BN15b}):

$\wH{m}{p}{q} = \set{G : G \arrowsv \uni{m}{p} \mbox{ and } \omega(G) < q}$.

$\wFv{m}{p}{q} = \min\set{\abs{\V(G)} : G \in \wH{m}{p}{q}}$.

To prove the upper bound in Theorem \ref{theorem: F_v(a_1, ..., a_s; m - 1) leq m + 12, max set(a_1, ..., a_s) = 7} we shall use the following results from \cite{BN15a}.

\begin{theorem}
	\label{theorem: F_v(2_(m - p), p; q) leq F_v(a_1, ..., a_s; q) leq wFv(m)(p)(q)}
	\cite{BN15a}
	Let $a_1, ..., a_s$ be positive integers and let $m$ and $p$ be defined by (\ref{equation: m and p}), $q > p$. Then
	\begin{equation*}
	F_v(2_{m - p}, p; q) \leq F_v(a_1, ..., a_s; q) \leq \wFv{m}{p}{q}.
	\end{equation*}
\end{theorem}

\begin{theorem}
	\label{theorem: wFv(m)(p)(m - m_0 + q) leq wFv(m_0)(p)(q) + m - m_0}
	\cite{BN15a}
	Let $m$, $m_0$, $p$ and $q$ be positive integers, $m \geq m_0$ and $q > \min\set{m_0, p}$. Then
	\begin{equation*}
	\wFv{m}{p}{m - m_0 + q} \leq \wFv{m_0}{p}{q} + m - m_0.
	\end{equation*}
\end{theorem}

\subsection*{Proof of Theorem \ref{theorem: F_v(a_1, ..., a_s; m - 1) leq m + 12, max set(a_1, ..., a_s) = 7}}

The lower bound in Theorem \ref{theorem: F_v(a_1, ..., a_s; m - 1) leq m + 12, max set(a_1, ..., a_s) = 7} follows from Theorem \ref{theorem: F_v(2, 2, 7; 8) = 20} and Theorem \ref{theorem: F_v(a_1, ..., a_s; m - 1) geq m + p + 3} To prove the upper bound, we shall use the graph $G_6 \in \mH(3, 7; 8) \cap \mH(4, 6; 8) \cap \mH(5, 5; 8)$ (see Figure \ref{figure: wH(9)(7)(8)(21)}) obtained by adding one vertex to the graph $G_5 \in \mH(2, 2, 7; 8; 20)$ (see Figure \ref{figure: H(2, 2, 7; 8; 20)}). Using (\ref{equation: G arrowsv (a_1, ..., a_s) Rightarrow G arrowsv (a_1, ..., a_(i - 1), t, a_i - t, a_(i + 1), ..., a_s)}), it is easy to prove that from $G_6 \arrowsv (3, 7)$, $G_6 \arrowsv (4, 6)$ and $G_6 \arrowsv (5, 5)$ it follows $G_6 \arrowsv \uni{9}{7}$. Therefore $G_6 \in \wH{9}{7}{8}$ and $\wFv{9}{7}{8} \leq 21$. Now from Theorem \ref{theorem: F_v(2_(m - p), p; q) leq F_v(a_1, ..., a_s; q) leq wFv(m)(p)(q)} and Theorem \ref{theorem: wFv(m)(p)(m - m_0 + q) leq wFv(m_0)(p)(q) + m - m_0} we derive
\begin{equation*}
F_v(a_1, ..., a_s; m - 1) \leq \wFv{m}{7}{m - 1} \leq \wFv{9}{7}{8} + m - 9 \leq m + 12.
\end{equation*}
\qed

Regarding the number $F_v(3, 7; 8)$, the following bounds were known:
$$18 \leq F_v(3, 7; 8) \leq 22.$$

The lower bound is true according to (\ref{equation: m + p + 2 leq F_v(a_1, ..., a_s; m - 1) leq m + 3p}) and the upper bound was proved in \cite{SXP09}. Using the results in this paper, we improve these bounds by proving the following

\begin{theorem}
\label{theorem: 20 leq F_v(3, 7; 8) leq 21}
$20 \leq F_v(3, 7; 8) \leq 21$.
\end{theorem}

\begin{proof}
The upper bound is true according to Theorem \ref{theorem: F_v(a_1, ..., a_s; m - 1) leq m + 12, max set(a_1, ..., a_s) = 7} and the lower bound follows from $F_v(3, 7; 8) \geq F_v(2, 2, 7; 8) = 20$.
\end{proof}

\begin{figure}[h]
	\centering
	\begin{subfigure}{\textwidth}
		\centering
		\includegraphics[height=332px,width=166px]{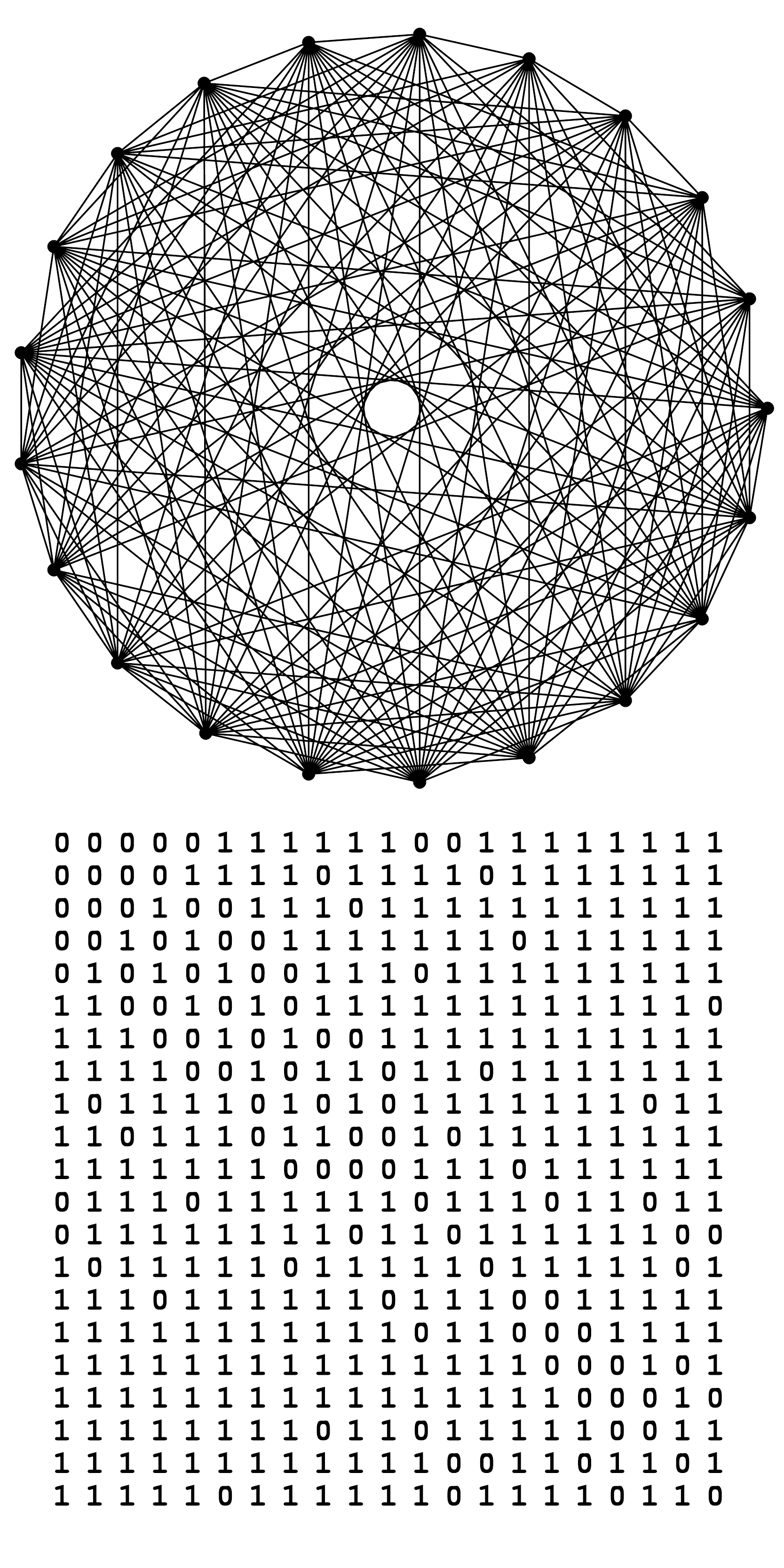}
		\caption*{\emph{$G_{6}$}}
		\label{figure: G_6}
	\end{subfigure}%
	\caption{21-vertex graph $\G_6 \in \wH{9}{7}{8}$}
	\label{figure: wH(9)(7)(8)(21)}
\end{figure}


\bibliographystyle{plain}

\bibliography{main}


\clearpage

\appendix

\section{Results of computations}

\begin{table}[!h]
	\centering
	\begin{tabular}{| p{3cm} | p{2cm} | p{2cm} | p{2cm} |}
		\hline
		set									& independence number	& maximal graphs 		& $(+K_6)$- graphs	\\
		\hline
		$\mH(2, 6; 7; 13)$			& $\leq 4$				& 1						& 1					\\
		$\mH(2, 2, 6; 7; 17)$		& $= 4$					& 0						&					\\
		\hline
		$\mH(3; 7; 6)$				& $\leq 3$				& 1						& 2					\\
		$\mH(4; 7; 8)$				& $\leq 3$				& 2						& 12				\\
		$\mH(5; 7; 10)$				& $\leq 3$				& 6						& 274				\\
		$\mH(6; 7; 12)$				& $\leq 3$				& 37					& 78 926			\\
		$\mH(2, 6; 7; 14)$			& $\leq 3$				& 20					& 5 291				\\
		$\mH(2, 2, 6; 7; 17)$		& $= 3$					& 0						&					\\
		\hline
		$\mH(3; 7; 7)$				& $\leq 2$				& 1						& 3					\\
		$\mH(4; 7; 9)$				& $\leq 2$				& 2						& 22				\\
		$\mH(5; 7; 11)$				& $\leq 2$				& 5						& 468				\\
		$\mH(6; 7; 13)$				& $\leq 2$				& 24					& 97 028			\\
		$\mH(2, 6; 7; 15)$			& $\leq 2$				& 473					& 10 018 539		\\
		$\mH(2, 2, 6; 7; 17)$		& $= 2$					& 1						&					\\
		\hline
		$\mH(2, 2, 6; 7; 17)$		& 						& 1						&					\\
		\hline
	\end{tabular}
	\caption{Steps in finding all maximal graphs in $\mH(2, 2, 6; 7; 17)$}
	\label{table: finding all graphs in H(2, 2, 6; 7; 17)}
\end{table}

\begin{table}[!b]
	\centering
	\begin{tabular}{| p{3cm} | p{2cm} | p{2cm} | p{2cm} |}
		\hline
		set									& independence number	& maximal graphs 		& $(+K_6)$- graphs	\\
		\hline
		$\mH(2, 6; 7; 13)$			& $\leq 5$				& 1						& 1					\\
		$\mH(3, 6; 7; 18)$			& $= 5$					& 0						&					\\
		\hline
		$\mH(3; 7; 6)$				& $\leq 4$				& 1						& 2					\\
		$\mH(4; 7; 8)$				& $\leq 4$				& 2						& 13				\\
		$\mH(5; 7; 10)$				& $\leq 4$				& 7						& 317				\\
		$\mH(6; 7; 12)$				& $\leq 4$				& 50					& 102 387			\\
		$\mH(2, 6; 7; 14)$			& $\leq 4$				& 20					& 5 293				\\
		$\mH(2, 2, 6; 7; 18)$		& $= 4$					& 0						&					\\
		\hline
		$\mH(3; 7; 7)$				& $\leq 3$				& 1						& 4					\\
		$\mH(4; 7; 9)$				& $\leq 3$				& 3						& 45				\\
		$\mH(5; 7; 11)$				& $\leq 3$				& 12					& 3 071				\\
		$\mH(6; 7; 13)$				& $\leq 3$				& 168					& 4 691 237			\\
		$\mH(2, 6; 7; 15)$			& $\leq 3$				& 1627					& 70 274 176		\\
		$\mH(2, 2, 6; 7; 18)$		& $= 3$					& 308					&					\\
		\hline
		$\mH(3; 7; 8)$				& $\leq 2$				& 1						& 8					\\
		$\mH(4; 7; 10)$				& $\leq 2$				& 3						& 82				\\
		$\mH(5; 7; 12)$				& $\leq 2$				& 10					& 5 057				\\
		$\mH(6; 7; 14)$				& $\leq 2$				& 96					& 2 799 416			\\
		$\mH(2, 6; 7; 16)$			& $\leq 2$				& 7509					& 920 112 878		\\
		$\mH(2, 2, 6; 7; 18)$		& $= 2$					& 84					&					\\
		\hline
		$\mH(2, 2, 6; 7; 18)$		& 						& 392					&					\\
		\hline
	\end{tabular}
	\caption{Steps in finding all maximal graphs in $\mH(2, 2, 6; 7; 18)$}
	\label{table: finding all graphs in H(2, 2, 6; 7; 18)}
\end{table}

\begin{table}[!h]
	\centering
	\begin{tabular}{| p{3cm} | p{2cm} | p{2cm} | p{2cm} |}
		\hline
		set									& independence number	& maximal graphs 		& $(+K_7)$- graphs	\\
		\hline
		$\mH(3, 6; 8; 14)$			& $\leq 4$				& 1						& 1					\\
		$\mH(2, 3, 6; 8; 18)$		& $= 4$					& 0						&					\\
		\hline
		$\mH(4; 8; 7)$				& $\leq 3$				& 1						& 2					\\
		$\mH(5; 8; 9)$				& $\leq 3$				& 2						& 12				\\
		$\mH(6; 8; 11)$				& $\leq 3$				& 6						& 276				\\
		$\mH(2, 6; 8; 13)$			& $\leq 3$				& 37					& 79 749			\\
		$\mH(3, 6; 8; 15)$			& $\leq 3$				& 21					& 3 458				\\
		$\mH(2, 3, 6; 8; 18)$		& $= 3$					& 0						&					\\
		\hline
		$\mH(4; 8; 8)$				& $\leq 2$				& 1						& 3					\\
		$\mH(5; 8; 10)$				& $\leq 2$				& 2						& 22				\\
		$\mH(6; 8; 12)$				& $\leq 2$				& 5						& 489				\\
		$\mH(2, 6; 8; 14)$			& $\leq 2$				& 25					& 119 126			\\
		$\mH(3, 6; 8; 16)$			& $\leq 2$				& 509					& 3 582 157			\\
		$\mH(2, 3, 6; 8; 18)$		& $= 2$					& 0						&					\\
		\hline
		$\mH(2, 3, 6; 8; 18)$		& 						& 0						&					\\
		\hline
	\end{tabular}
	\caption{Steps in finding all maximal graphs in $\mH(2, 3, 6; 8; 18)$}
	\label{table: finding all graphs in mathcal(H)(2, 3, 6; 8; 18)}
\end{table}

\begin{table}[!b]
	\centering
	\begin{tabular}{| p{3cm} | p{2cm} | p{2cm} | p{2cm} |}
		\hline
		set									& independence number	& maximal graphs 		& $(+K_8)$- graphs	\\
		\hline
		$\mH(2, 3, 6; 9; 15)$		& $\leq 4$				& 1						& 1					\\
		$\mH(2, 2, 3, 6; 9; 19)$	& $= 4$					& 0						&					\\
		\hline
		$\mH(5; 9; 8)$				& $\leq 3$				& 1						& 2					\\
		$\mH(6; 9; 10)$				& $\leq 3$				& 2						& 12				\\
		$\mH(2, 6; 9; 12)$			& $\leq 3$				& 6						& 277				\\
		$\mH(3, 6; 9; 14)$			& $\leq 3$				& 37					& 79 901			\\
		$\mH(2, 3, 6; 9; 16)$		& $\leq 3$				& 21					& 3 459				\\
		$\mH(2, 2, 3, 6; 9; 19)$	& $= 3$					& 0						&					\\
		\hline
		$\mH(5; 9; 9)$				& $\leq 2$				& 1						& 3					\\
		$\mH(6; 9; 11)$				& $\leq 2$				& 2						& 22				\\
		$\mH(2, 6; 9; 13)$			& $\leq 2$				& 5						& 496				\\
		$\mH(3, 6; 9; 15)$			& $\leq 2$				& 25					& 121 499			\\
		$\mH(2, 3, 6; 9; 17)$		& $\leq 2$				& 512					& 3 585 530			\\
		$\mH(2, 2, 3, 6; 9; 19)$	& $= 2$					& 0						&					\\
		\hline
		$\mH(2, 2, 3, 6; 9; 19)$	& 						& 0						&					\\
		\hline
	\end{tabular}
	\caption{Steps in finding all maximal graphs in $\mH(2, 2, 3, 6; 9; 19)$}
	\label{table: finding all graphs in mathcal(H)(2, 2, 3, 6; 9; 19)}
\end{table}

\begin{table}[!h]
	\centering
	\begin{tabular}{| p{3cm} | p{2cm} | p{2cm} | p{2cm} |}
		\hline
		set										& independence number	& maximal graphs 		& $(+K_9)$- graphs	\\
		\hline
		$\mH(2, 2, 3, 6; 10; 16)$		& $\leq 4$				& 1						& 1					\\
		$\mH(2, 2, 2, 3, 6; 10; 20)$	& $= 4$					& 0						&					\\
		\hline
		$\mH(6; 10; 9)$					& $\leq 3$				& 1						& 2					\\
		$\mH(2, 6; 10; 11)$				& $\leq 3$				& 2						& 12				\\
		$\mH(3, 6; 10; 13)$				& $\leq 3$				& 6						& 277				\\
		$\mH(2, 3, 6; 10; 15)$			& $\leq 3$				& 37					& 79 934			\\
		$\mH(2, 2, 3, 6; 10; 17)$		& $\leq 3$				& 21					& 3 459				\\
		$\mH(2, 2, 2, 3, 6; 10; 20)$	& $= 3$					& 0						&					\\
		\hline
		$\mH(6; 10; 10)$				& $\leq 2$				& 1						& 3					\\
		$\mH(2, 6; 10; 12)$				& $\leq 2$				& 2						& 22				\\
		$\mH(3, 6; 10; 14)$				& $\leq 2$				& 5						& 498				\\
		$\mH(2, 3, 6; 10; 16)$			& $\leq 2$				& 25					& 121 864			\\
		$\mH(2, 2, 3, 6; 10; 18)$		& $\leq 2$				& 512					& 3 585 546			\\
		$\mH(2, 2, 2, 3, 6; 10; 20)$	& $= 2$					& 0						&					\\
		\hline
		$\mH(2, 2, 2, 3, 6; 10; 20)$	& 						& 0						&					\\
		\hline
	\end{tabular}
	\caption{Steps in finding all maximal graphs in $\mH(2, 2, 2, 3, 6; 10; 20)$}
	\label{table: finding all graphs in mathcal(H)(2, 2, 2, 3, 6; 10; 20)}
\end{table}

\begin{table}[!b]
	\centering
	\begin{tabular}{| p{3cm} | p{2cm} | p{2cm} | p{2cm} |}
		\hline
		set									& independence number	& maximal graphs 		& $(+K_7)$- graphs	\\
		\hline
		$\mH(2, 7; 8; 15)$			& $\leq 4$				& 1						& 1					\\
		$\mH(2, 2, 7; 8; 19)$		& $= 4$					& 0						&					\\
		\hline
		$\mH(4; 8; 8)$				& $\leq 3$				& 1						& 4					\\
		$\mH(5; 8; 10)$				& $\leq 3$				& 3						& 45				\\
		$\mH(6; 8; 12)$				& $\leq 3$				& 12					& 3 104				\\
		$\mH(7; 8; 14)$				& $\leq 3$				& 169					& 4 776 518			\\
		$\mH(2, 7; 8; 16)$			& $\leq 3$				& 34					& 22 896			\\
		$\mH(2, 2, 7; 8; 19)$		& $= 3$					& 0						&					\\
		\hline
		$\mH(4; 8; 9)$				& $\leq 2$				& 1						& 8					\\
		$\mH(5; 8; 11)$				& $\leq 2$				& 3						& 84				\\
		$\mH(6; 8; 13)$				& $\leq 2$				& 10					& 5 394				\\
		$\mH(7; 8; 15)$				& $\leq 2$				& 102					& 4 984 994			\\
		$\mH(2, 7; 8; 17)$			& $\leq 2$				& 2760					& 380 361 736		\\
		$\mH(2, 2, 7; 8; 19)$		& $= 2$					& 0						&					\\
		\hline
		$\mH(2, 2, 7; 8; 19)$		& 						& 0						&					\\
		\hline
	\end{tabular}
	\caption{Steps in finding all maximal graphs in $\mH(2, 2, 7; 8; 19)$}
	\label{table: finding all graphs in H(2, 2, 7; 8; 19)}
\end{table}

\end{document}